\documentclass[12pt]{amsart}
\usepackage{amsmath}
\usepackage[backref=page]{hyperref}
\usepackage{tikz,float,caption}

\usepackage[margin=1.25 in,top=1.2in, bottom=1.4 in]{geometry}

\usepackage[bb=ams,scr=euler]{mathalpha}

\usepackage{setspace}
\usepackage{enumitem}
\setenumerate{
  wide=10pt,
  leftmargin=0pt,
  labelwidth=*,
  label=(\roman*),
  topsep=0pt,
  listparindent=0pt,
  parsep=4pt}

\usepackage{thmtools}
\declaretheoremstyle[headfont=\normalsize\normalfont\bfseries,notefont=\mdseries, notebraces={(}{)},bodyfont=\normalfont,postheadspace=0.5em]{basicstyle}
\declaretheoremstyle[headfont=\normalsize\normalfont\bfseries,notefont=\mdseries,
notebraces={(}{)},bodyfont=\normalfont\itshape,postheadspace=0.5em]{italstyle}

\declaretheorem[style=italstyle,name=Theorem,numberwithin=section]{theorem}
\declaretheorem[style=italstyle,name=Corollary,sibling=theorem]{cor}

\declaretheorem[style=italstyle,name=Proposition,sibling=theorem]{prop}
\declaretheorem[style=italstyle,name=Lemma,sibling=theorem]{lemma}
\declaretheorem[style=basicstyle,name=Proof,numbered=no,qed=$\square$]{preproof}
\renewenvironment{proof}{\preproof}{\endpreproof}

\makeatletter
\newcommand{\abs}[1]{\left|#1\right|}
\newcommand{\bd}{\partial}
\newcommand{\cl}[1]{\overline{#1}}
\newcommand{\C}{\mathbb{C}}
\renewcommand{\d}{\mathrm{d}}

\newcommand{\intprod}{\mathbin{{\tikz{\draw(-0.1,0)--(0.1,0)--(0.1,0.2)}\hspace{0.5mm}}}}
\newcommand{\ip}[1]{\left\langle#1\right\rangle}
\newcommand{\norm}[1]{\left\lVert#1\right\rVert}

\newcommand{\pr}{\mathrm{pr}}
\newcommand{\R}{\mathbb{R}}
\newcommand{\set}[1]{\left\{#1\right\}}
\newcommand{\Z}{\mathbb{Z}}

\renewcommand\section{\@startsection{section}{1}{0pt}{-3.5ex \@plus -1ex \@minus -.2ex}{2.3ex \@plus.2ex}{\centering\itshape}}
\renewcommand{\subsection}{\@startsection{subsection}{2}%
  \z@{.5\linespacing\@plus.7\linespacing}{-.5em}%
  {\normalfont\itshape}}

\makeatother

\title[Bordism classes of loops and Floer's equation]{Bordism classes of loops and Floer's equation in cotangent bundles}
\author{Filip Bro\'ci\'c}
\author{Dylan Cant}
\date{\today}
\begin{document}
\maketitle

\begin{abstract}
  For each representative $\mathfrak{B}$ of a bordism class in the free loop space of a manifold, we associate a moduli space of finite length Floer cylinders in the cotangent bundle. The left end of the Floer cylinder is required to be a lift of one of  the loops in $\mathfrak{B}$, and the right end is required to lie on the zero section. Under certain assumptions on the Hamiltonian functions, the length of the Floer cylinder is a smooth proper function, and evaluating the level sets at the right end produces a family of loops cobordant to $\mathfrak{B}$. The argument produces arbitrarily long Floer cylinders with certain properties. We apply this to prove an existence result for 1-periodic orbits of certain Hamiltonian systems in cotangent bundles, and also to estimate the relative Gromov width of starshaped domains in certain cotangent bundles. The moduli space is similar to moduli spaces considered in \cite{abbondandolo_schwarz_2,abouzaid_based_loop} for Tonelli Hamiltonians. The Hamiltonians we consider are not Tonelli, but rather of ``contact-type'' in the symplectization end.
\end{abstract}

\section{Introduction}
\label{sec:introduction}

Let $M$ be a compact $n$-dimensional manifold. This paper concerns the relationship between the free loop space of $M$ and moduli spaces of solutions to Floer's equation in $T^{*}M$. Such a relationship is well-known, and the most famous example is that the symplectic homology of $T^{*}M$ is isomorphic to the homology of the free loop space. This is originally due to \cite{viterbo_cotangent_1,viterbo_cotangent_2}, and was deduced using the generating function approach to symplectic homology of \cite{traynor_GF_1}. Subsequently, other proofs were discovered, notably the heat-flow approach of \cite{salamon_weber}, and the Morse theory approach of \cite{abbondandolo_schwarz, abbondandolo_portaluri_schwarz, abbondandolo_schwarz_2, abbondandolo_schwarz_3} and \cite{abouzaid_based_loop,abouzaid_monograph}.

This paper concerns a different approach and relates bordism classes in the free loop space to moduli spaces of Floer cylinders. As application, we are able to prove an existence result for $1$-periodic orbits of a large class of Hamiltonian systems, Theorem \ref{theorem:A}, and estimate the relative Gromov width of star-shaped domains in certain cotangent bundles (relative the zero section), Theorem \ref{theorem:B}. Our approach also has the advantage of circumventing the need to consider indices, orientations of moduli spaces, or any chain complexes at all.

\subsection{Applications}
\label{sec:applications}
We begin by explaining the applications of our method, and then describe the main technical result of the paper in \S\ref{sec:definition_of_moduli_space}.

\subsubsection{An existence theorem for contact-at-infinity Hamiltonian systems}
\label{sec:an-existence-theorem}
A Hamiltonian system $H_{t}:T^{*}M\to \R$ is called \emph{contact-at-infinity} if it satisfies:
\begin{equation*}
  H_{t}(q,e^{s}p)=e^{s}H_{t}(q,p),
\end{equation*}
for all $s\ge 0$ and all $(q,p)$ outside of a compact set. The significance of such Hamiltonians is that their flow $\varphi_{t}$ is equivariant with respect to the Liouville flow, outside of a compact set. It is well-known that quotient space:
$$Y=(T^{*}M-\text{zero section})/\R$$
inherits a canonical contact structure (as the ideal boundary of a Liouville manifold), and the flow of contact-at-infinity Hamiltonian induces a contact isotopy of $Y$.

As a special case, one can fix a starshaped domain $\Omega\subset T^{*}M$, as defined in \S\ref{sec:star-shaped domains}, and define $r_{\Omega}:T^{*}M\to [0,\infty)$ by requiring $r_{\Omega}(\bd\Omega)=1$ and $r_{\Omega}(q,e^{s}p)=e^{s}r_{\Omega}(q,p)$. This function $r_{\Omega}$ will not be smooth along the zero section, however, one can consider Hamiltonian systems $H_{t}$ so that $H_{t}=ar_{\Omega}$ outside a compact set. Such Hamiltonians are contact-at-infinity, and can be used to define a version of symplectic homology for the domain $\Omega$.

More generally, one can consider Hamiltonians which satisfy $H_{t}=(h_{t}\circ \pi) r_{\Omega}$ outside a compact set, where $h_{t}:Y\to \R$ is a 1-periodic family of functions and $\pi$ is the quotient projection $(T^{*}M-\text{zero section})\to Y$. Such Hamiltonian systems are considered in, e.g., \cite{merry_ulja,uljarevic}.

The first application of our methods is the following existence result:
\begin{theorem}\label{theorem:A}
  Any contact-at-infinity Hamiltonian $H_{t}$ which is non-positive on a fiber $T^{*}M_{q_{0}}$ will have at least one contractible 1-periodic orbit (we only require non-positivity outside of some starshaped domain).
\end{theorem}
A more general result is proved in Theorem \ref{theorem:exist_1}.

We remark that there are contact-at-infinity Hamiltonians which have no 1-periodic orbits. For example, one can take a canonical transformation lifting an isotopy in the base without 1-periodic orbits (on a manifold with vanishing Euler characteristic). In this case, the generating Hamiltonians are linear and non-vanishing on each fiber.

\subsubsection{An estimate on the relative Gromov width of starshaped domains}
\label{sec:an-estimate-gromov}

Recall that the \emph{relative Gromov width} of a pair $(W,L)$, where $W$ is a symplectic manifold and $L$ is a Lagrangian submanifold, is the maximal $a>0$ for which there is a symplectic embedding $B(a)\to W$ which maps $B(a)\cap \R^{n}$ onto $L$, where $B(a)$ is the ball of symplectic capacity $a$. This notion is considered in \cite{barraud_cornea_1,barraud_cornea_2,biran_cornea_CRM,biran_cornea_geom_topol,biran_cornea_shelukhin,brocic_1}.

Our method produces an estimate for the relative Gromov width of starshaped domains $\Omega\subset T^{*}M$ under the assumption that $M$ has an $\R/\Z$-action $\zeta$ with non-contractible orbits.

\begin{theorem}\label{theorem:B}
  Let $M$ be a space with an $\R/\Z$-action with non-contractible orbits. Suppose that $\Omega$ is a star-shaped domain in $T^{*}M$. The relative Gromov width of $(\Omega,M)$ is at most:
  \begin{equation*}
    2\sup\set{\int_{\R/\Z} \gamma^{*}\lambda:\gamma(t)\text{ lies in $\Omega$ and lifts }\zeta(t,q)\text{ for some }q\in M}\in (0,\infty),
  \end{equation*}
  where $\zeta:\R/\Z\times M\to M$ is the $\R/\Z$-action.
\end{theorem}

One example is the manifold $M=Q\times \R/\Z$ where the action is $\zeta_{t}(q,\tau)=(q,\tau+t)$. Interestingly, the same hypothesis appears in \cite[Corollary 1.3]{irie_hz}, where the Hofer-Zehnder capacity of starshaped domains in cotangent bundles is considered.

In \S\ref{sec:sharpness}, the estimate from Theorem \ref{theorem:B} is proved to be sharp in some cases.

\subsection{Definition of the moduli space}
\label{sec:definition_of_moduli_space}
We now describe the method used to prove Theorems \ref{theorem:A} and \ref{theorem:B}. Fix admissible Hamiltonian $H_{t}$ and complex structure $J_{t}$ on the cotangent bundle $T^{*}M$ of a compact manifold $M^{n}$ (see \S\ref{sec:admissible_data} for the definition of admissibility). Also fix a smooth map $\mathfrak{B}:P\times \R/\Z\to M,$ considered as a family of loops $\mathfrak{B}(\pi,t)$ parametrized by a compact manifold $P^{d}$ without boundary. Associated to these choices define:
\begin{equation*}
  \mathscr{M}(H_{t},J_{t},\mathfrak{B})=\left\{(u,\pi,\ell):
    \begin{aligned}
      &u:[0,\ell]\times \R/\Z\to T^{*}M\text{ for }\ell>0,\\
      &\bd_{s}u+J_{t}(u)(\bd_{t}u-X_{H_{t}}(u))=0,\\
      &\pr\circ u(0,t)=\mathfrak{B}(\pi,t)\text{ and }u(\ell,t)\in M
    \end{aligned}
  \right\}.
\end{equation*}
In words, $\mathscr{M}(H_{t},J_{t},\mathfrak{B})$ is the space of finite length Floer cylinders whose left end is a lift of a loop in $\mathfrak{B}$, and whose right end lies on the zero section. A similar moduli space is considered in \cite[\S 5]{abouzaid_based_loop}, where $\mathfrak{B}(\pi,t)$ is some unstable manifold associated to a Riemannian energy functional on the based loop space.

\begin{figure}[H]
  \centering
  \begin{tikzpicture}[yscale=0.8]
    \draw[blue,rotate=-20] (-1,2) circle (0.5 and 1);
    \path[rotate=-20] (-1,3) arc (90:270:0.5 and 1) coordinate[pos=0] (A) coordinate[pos=1] (B);
    \path[rotate=-20] (-1,3) arc (90:180:0.5 and 1) coordinate[pos=1] (X);
    \node at (X) [left] {lift of a loop in $\mathfrak{B}$};
    \node at (X.east) [shift={(5,0)}] {\phantom{lift of a loop in $\mathfrak{B}$}};

    \draw[blue,rotate=20,shift={(0,-.6)}] (2,0) circle (1 and 0.5);
    \path[rotate=20,shift={(0,-.6)}] (3,0) arc (0:180:1 and 0.5) coordinate[pos=0] (C) coordinate[pos=1] (D);
    \draw (-2,-1)--(-1,1)--(4,1)--(3,-1)node[right]{$M$}--cycle;

    \draw[blue] (A)to[out=-10,in=100](C) (B)to[out=-10,in=100](D);
    \node[blue] at (1.3,1.5) {$u$};
  \end{tikzpicture}
  \caption{The left end of the Floer cylinder $u$ traces out a lift of one of the loops $\mathfrak{B}(\pi,t)$. The right end lies on the zero section.}
  \label{fig:1}
\end{figure}
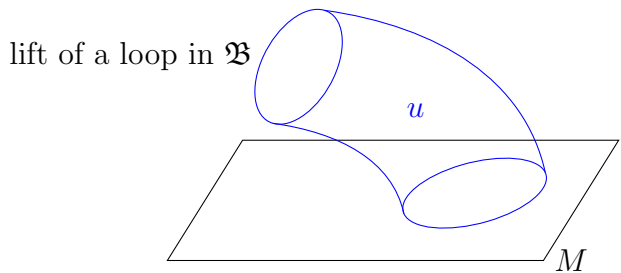
Before stating the main result, Theorem \ref{theorem:main}, we introduce a few necessary terms.

\subsubsection{Admissible data}
\label{sec:admissible_data}
Admissibility of $(H_{t},J_{t})$ is defined in terms of the behaviour at infinity. Writing $\rho_{\sigma}(q,p)=(q,e^{\sigma}p)$ for the \emph{Liouville flow}, we say $H_{t},J_{t}$ is \emph{admissible for a starshaped domain} $\Omega$ (e.g., a unit disk bundle for a metric) if $H_{t}\circ \rho_{\sigma}(z)=e^{\sigma}H_{t}(z)$ and $J_{\rho_{\sigma}(z)}\d\rho_{\sigma,z}=\d\rho_{\sigma,z}J_{t,z}$ hold for all $\sigma\ge 0$ and $z\not\in \Omega$. Moreover, we require that $H_{t}$ has no 1-periodic orbits outside $\Omega$. Briefly, $H_{t}$ is $1$-homogeneous and $J_{t}$ is translation invariant in the ends. See \S\ref{sec:admissible_hamiltonians} for more details.

\subsubsection{Bordism classes of loops}
\label{sec:bordism_intro}
$\mathfrak{B}_{1}:P_{1}\times \R/\Z\to M$ and $\mathfrak{B}_{2}:P_{2}\times \R/\Z\to M$ are \emph{cobordant} provided that the compact manifolds $P_{1},P_{2}$ are cobordant, say $P_{1}\sqcup P_{2}=Q$, and there is a function $Q\times \R/\Z\to M$ which restricts to $\mathfrak{B}_{1},\mathfrak{B}_{2}$ at the two boundary components. Notice that the restricted evaluation maps $\mathfrak{B}_{1}(-,t_{0})$ and $\mathfrak{B}_{1}(-,t_{0})$ are also cobordant maps of $P_{1},P_{2}$ into $M$.

\subsubsection{Sufficiently negative Hamiltonians}
\label{sec:sufficiently_negative_hamiltonians}
Let $\mathfrak{B}(\pi,t)$ be a family of loops in $M$, and let $H_{t}$ be an admissible Hamiltonian for a star-shaped domain $\Omega\subset T^{*}M$. Say that $H_{t}$ is \emph{sufficiently negative} for $\mathfrak{B},\Omega$ provided that:
\begin{equation}\label{eq:suff_neg}
  \gamma(t)\not\in \Omega\text{ for all $t$ and }\mathrm{pr}(\gamma(t))=\mathfrak{B}(\pi,t)\text{ for some }\pi\implies\mathscr{A}(\gamma)\le 0.
\end{equation}
Here $\mathscr{A}(\gamma)=\int H_{t}(\gamma)-\gamma^{*}\lambda$, see \S\ref{sec:hamiltonian_action}. Since $\mathscr{A}(\rho_{\sigma}(\gamma))=e^{\sigma}\mathscr{A}(\gamma)$ for loops satisfying the hypothesis of \eqref{eq:suff_neg}, it is clear that the sufficiently negative condition is necessary to ensure an a priori bound on the action of lifts of $\mathfrak{B}$. The condition is also sufficient, as proved in Proposition \ref{prop:suff_neg_sufficient}.

\subsection{Statement of the main result}
\label{sec:statement_of_main_result}
Let $(H_{t},J_{t})$ be admissible for a starshaped domain and $\mathfrak{B}$ a family of loops, let $\mathscr{M}:=\mathscr{M}(H_{t},J_{t},\mathfrak{B})$, and let $\ell:\mathscr{M}\to (0,\infty)$ be the coordinate projection $(u,\pi,\ell)\mapsto \ell$. Abbreviate $d=\dim P$.
\begin{theorem}\label{theorem:main}
  If $H_{t}$ is sufficiently negative for $\mathfrak{B}$, $\Omega$, then $\ell:\mathscr{M}\to (0,\infty)$ is proper. Moreover, there exist $C^{\infty}$ small perturbations $\delta_{t}$ supported in $\Omega$ so that $$\mathscr{M}':=\mathscr{M}(H_{t}+\delta_{t},J_{t},\mathfrak{B})$$ is a smooth manifold of dimension $d+1$. The map $\ell$ is smooth, and for any regular value $\ell_{0}$, the fiber $\mathscr{M}'(\ell_{0}):=\mathscr{M}'\cap \set{\ell=\ell_{0}}$ is a compact $d$-dimensional manifold. The evaluation map:
  \begin{equation*}
    \mathrm{ev}_{\ell_{0}}:(u,\pi,\ell_{0},t)\in \mathscr{M}'(\ell_{0})\times \R/\Z\mapsto u(\ell_{0},t)\in M,
  \end{equation*}
  is smooth and cobordant (as a bordism class of loops) to the original class $\mathfrak{B}$. Moreover, the projection map $\pi:\mathscr{M}'(\ell_{0})\to P$ has mapping degree $1$. In particular $\mathscr{M}'(\ell_{0})$ is never empty.
\end{theorem}
The proof of the theorem is the main goal of the paper. The compactness statements are proved in \S\ref{sec:compactness}, while the transversality and gluing statements are proved in \S\ref{sec:transversality_gluing}.

One can interpret the main result as an \emph{evolution equation}, in the sense that one begins with some initial condition $\mathfrak{B}$ and evolves it using the non-linear Floer equation, obtaining cobordant families of loops. Unlike the typical \emph{parabolic} evolution equations (such as heat flow), the Floer equation is an \emph{elliptic} equation, and it is well-known that the flow is not well-posed. The sufficiently negative assumption in Theorem \ref{theorem:main} is necessary to control the energy of the Floer cylinders; without this control one would not be able to conclude the existence of solutions for large values of the parameter $\ell$.

\subsubsection{Applying the main theorem}
\label{sec:applying}
Observe that, if $(u_{n},\pi_{n},\ell_{n})\in \mathscr{M}(H_{t}+\delta^{n}_{t},J_{t},\mathfrak{B})$, with $\delta^{n}_{t}\to 0$ and $\ell_{n}\to \infty$, then a standard compactness-up-to-breaking argument in Floer theory implies the existence of finite-energy solutions to the problem:
\begin{equation}\label{eq:compactness-up-to-breaking}
  \left\{
    \begin{aligned}
      &u:(-\infty,0]\times \R/\Z\to T^{*}M\\
      &\bd_{s}u+J_{t}(u)(\bd_{t}u-X_{H_{t}})=0\\
      &u(0,t)\in M,
    \end{aligned}
  \right.
\end{equation}
The finiteness of energy implies that such solutions will be asymptotic at their negative end to 1-periodic orbits of $X_{H_{t}}$ as in Figure \ref{fig:3}. See \S\ref{sec:length_diverging} for more details.

\begin{figure}[H]
  \centering
  \begin{tikzpicture}[xscale=1.5,yscale=0.75]
    \draw (0,0) arc (-90:270:0.2 and 1)coordinate[pos=0.75](B)--+(3,0) arc (-90:90:0.2 and 1)coordinate[pos=0.5](A) coordinate[pos=1](Y)--+(-3,0);
    \draw[dashed] (Y) arc (90:270:0.2 and 1);
    \node at (B)[left]{orbit of $X_{H_{t}}$};
    \node at (A)[right]{lies on zero section};
  \end{tikzpicture}
  \caption{Compactness up-to-breaking implies the existence of a half-infinite Floer cylinder whose left end is asymptotic to a Hamiltonian orbit, and whose right end lies on the zero section.}
  \label{fig:3}
\end{figure}
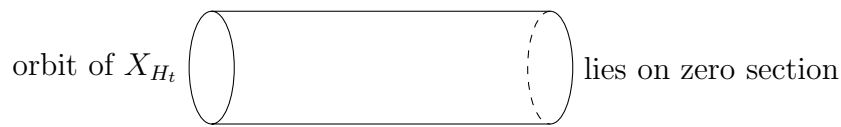

The moduli space $\mathscr{M}(H_{t}+\delta_{t},J_{t},\mathfrak{B})\cap \set{\ell=\ell_{0}}$ must be non-empty for any regular value $\ell_{0}$. Thus take $\ell_{0}$ as large as desired, and apply the compactness-up-to-breaking argument to conclude the existence of solutions to \eqref{eq:compactness-up-to-breaking}. The following a priori energy bound controls the energy of the solution.

\subsubsection{A priori energy bound}
\label{sec:priori-energy-bound}
The following result is important in our applications.
\begin{lemma}\label{lemma:apriori_bound_kappa}
  Let $H_{t},J_{t}$ be admissible for $\Omega$ and suppose $H_{t}$ is sufficiently negative for $\mathfrak{B},\Omega$. Then there is an a priori energy bound for $u\in \mathscr{M}(H_{t},J_{t},\mathfrak{B})$:
  \begin{equation*}
    E(u)=\int \norm{\bd_{s}u}^{2}_{J_{t}}\d s\d t\le \int_{0}^{1} \max_{\Omega} H_{t}-\min_{M}H_{t}\,\d t-\kappa(\mathfrak{B},\Omega),
  \end{equation*}
  where:
  \begin{equation*}
    \kappa(\mathfrak{B},\Omega):=\inf\set{\int_{\R/\Z} \gamma^{*}\lambda:\gamma(t)\text{ is a lift of some }\mathfrak{B}(\pi,t)\text{ lying in }\Omega}\in (-\infty,0].
  \end{equation*}
\end{lemma}
\begin{proof}
  This is proved in \S\ref{sec:suff_neg}.
\end{proof}

\subsection{Proof of Theorem \ref{theorem:A}}
\label{sec:existence_orbits}

Let $P$ be a connected manifold, and fix some family $\mathfrak{B}:P\times \R/\Z\to M$. Since $P$ is connected, $\mathfrak{B}$ determines a free homotopy class, denoted $[\mathfrak{B}]$.

Then we have the following existence result, similar to results in \cite{biran_polterovich_salamon,weber_noncontractible}. The non-positive quantity $\kappa(\mathfrak{B},\Omega)$ is defined in Lemma \ref{lemma:apriori_bound_kappa}.
\begin{theorem}\label{theorem:exist_1}
  If $H_{t}$ is admissible for $\Omega$ and $H_{t}<\kappa(\mathfrak{B},\Omega)$ holds on the set: $$\bd\Omega \cap \pr^{-1}(\mathfrak{B}(P\times\R/\Z)),$$ then $X_{H_{t}}$ has a $1$-periodic orbit in $\Omega$ whose projection to $M$ lies in the free homotopy class $[\mathfrak{B}]$.
\end{theorem}
\begin{proof}
  The idea is simple, show that $H_{t}$ is sufficiently negative for $\mathfrak{B},\Omega$ and apply the argument in \S\ref{sec:applying}. In fact, we will not work directly with $\mathfrak{B}$, but rather with a time reparametrization of $\mathfrak{B}$. To simplify the exposition, let us suppose that $P=\set{\pi}$ is a single point, and write $q(t):=\mathfrak{B}(\pi,t)$, $\kappa:=\kappa(\mathfrak{B},\Omega)$.

  A lift in $\Omega$ can be written as $\gamma(t)=(q(t),p(t))$ where $p(t)\in T^{*}M_{q(t)}$. Then:
  \begin{equation*}
    \gamma^{*}\lambda=\ip{p(t),q'(t)}\d t\ge \min\set{\ip{p,q'(t)}:p\in T^{*}M_{q(t)}\cap \Omega}\d t=:\Lambda(t)\d t.
  \end{equation*}
  The quantity $\Lambda(t)$ is simply the (negative) minimum of the linear function $\ip{-,q'(t)}$ over a star-shaped domain. It can be shown that:
  \begin{equation}\label{eq:kappa_lambda}
    \kappa=\int \Lambda(t)\d t.
  \end{equation}
  The fact that $\ge$ holds is clear, and $\le$ follows by picking smooth lifts $p(t)$ so that the pairing $\ip{p(t),q'(t)}$ is close to the theoretical minimum $\Lambda(t)$ for most $t\in [0,1]$, and then recalling that $\kappa$ is an infimum over all smooth lifts $p(t)$. In general, the minimizer of a linear function on a star-shaped domain does not depend continuously on the linear function or the star-shaped domain, but one can still get $\int \gamma^{*}\lambda$ arbitrarily close to $\int \Lambda(t)\d t$.

  Fix some $\delta>0$. By smoothly reparametrizing $q(t)$ in time, one can make $\Lambda(t)$ approximately constant, in the sense that $\kappa-\delta\le \Lambda(t)\le \kappa+\delta$ holds for all $t$. This should be thought of as a generalized ``constant speed'' parametrization. Henceforth replace $\mathfrak{B}$ by this reparametrized $\mathfrak{B}$.

  If $\delta$ is sufficiently small, then $H_{t}$ will be sufficiently negative for $\mathfrak{B},\Omega$. Indeed, if $\gamma(t)$ is a lift which remains outside of $\Omega$, then $\gamma(t)=(q(t),e^{f(t)}p(t))$ where $p(t)\in \bd\Omega$, hence:
  \begin{equation*}
    H_{t}(\gamma(t))\d t-\gamma^{*}\lambda\le e^{f(t)}(H_{t}(q,p)-\Lambda(t))\d t\le e^{f(t)}(H_{t}(q,p)-\kappa+\delta)\d t<0,
  \end{equation*}
  using that $H_{t}(q,p)-\kappa+\delta<0$ for all $(q,p)\in \Omega$ with $q\in \mathfrak{B}(P\times \R/\Z)$ for sufficiently small $\delta$. Thus the compactness argument outlined in \S\ref{sec:applying} can be performed, and one concludes a 1-periodic orbit of $X_{H_{t}}$ whose projection to $M$ is in the same free homotopy class of $\mathfrak{B}$. This completes the proof.
\end{proof}
As a corollary, one can use the constant map $\mathfrak{B}(\pi,t)=q_{0}$, with $P=\set{\pi}$, which has $\kappa=0$, to conclude Theorem \ref{theorem:A}.

\subsection{Proof of Theorem \ref{theorem:B}}
\label{sec:relative_gromov_width}

Let $\Omega$ be a starshaped domain in the cotangent bundle $T^{*}M$, and suppose that $M$ has an $\R/\Z$-action $\zeta_{t}:M\to M$ so that the orbits $t\mapsto \zeta(t,q)=\zeta_{t}(q)$ are non-contractible. There is a natural family of loops $\mathfrak{B}(\pi,t)=\zeta_{t}(\pi)$, where $P=M$. It is clear that $\pi\mapsto \mathfrak{B}(\pi,0)\in M$ has degree $1$, and this property is stable under cobordism (of families of loops). This gives an effective way to estimate the relative Gromov width of the pair $(\Omega,M)$.

\begin{proof}[of Theorem \ref{theorem:B}]
  See \cite{brocic_1} for similar arguments. Let $\mathfrak{B}(\pi,t)=\zeta_{t}(\pi)$, as above, First observe that:
  \begin{equation*}
    -\kappa(-\Omega,\mathfrak{B})=\sup\int\set{\int_{\R/\Z} \gamma^{*}\lambda:\gamma(t)\text{ lies in $\Omega$ and lifts }\mathfrak{B}(\pi,t)\text{ for some }\pi\in P}.
  \end{equation*}
  It is clear that $\Omega$ and $-\Omega$ have the same relative Gromov width, and hence it suffices to prove that the Gromov width of $\Omega$ is bounded by $-2\kappa(\Omega,\mathfrak{B})$, which is slightly more natural with our method.

  Let $\iota:B(a)\to \Omega$ be a symplectic embedding. By a limiting argument, we may suppose that $\iota$ extends smoothly to a map $B(a')\to \Omega$ for some $a'>a$. Let $J$ be a complex structure on $T^{*}M$ so that $J\d\iota=\d\iota J_{0}$ holds on the image of $\iota$ and is admissible for $\Omega$.

  Fix $\epsilon>0$, and pick $H$ so that $\max_{\Omega} H-\min_{M}H\le \epsilon$, $\min_{\bd\Omega}H\le -\epsilon^{-1}$, $\d H=0$ on $\iota(B(a))$, and $H_{t}$ is admissible for $\Omega$. It is not hard to construct such an $H$, one can take a smoothing of the function which equals $-\epsilon^{-1}$ on $\Omega$ which is extended to be admissible (i.e., $1$-homogeneously). For $\epsilon$ small enough, $H$ will be sufficiently negative for $\mathfrak{B}$.

  Applying Theorem \ref{theorem:main}, one concludes that, for a sequence $\ell_{n}\to \infty$ and a sequence of perturbations $\delta^{n}_{t}\to 0$ supported in $\Omega$, the moduli spaces:
  \begin{equation*}
    \mathscr{M}_{n}:=\mathscr{M}(H+\delta^{n}_{t},J,\mathfrak{B})\cap \set{\ell=\ell_{n}},
  \end{equation*}
  are compact manifolds, and the evaluation $\mathrm{ev}_{n}:\mathscr{M}_{n}\times \R/\Z\to M$ is cobordant to the $\R/\Z$-action map $M\times \R/\Z\to M$. In particular, the restricted evaluation map $\mathrm{ev}_{n}:\mathscr{M}_{n}\times \set{0}\to M$ has degree $1$. Therefore, for every $n$ we can find $(u_{n},\pi_{n},\ell_{n})$ so that $u_{n}(\ell_{n},0)=\iota(0)$, i.e., $u_{n}$ passes through the center of the ball.

  \begin{figure}[H]
    \centering
    \begin{tikzpicture}
      \draw (0,0) circle (2);
      \draw[dashed] (2,0) arc (0:180:2 and 0.5);
      \draw (-2,0) arc (180:360:2 and 0.5);
      \draw[blue] (0.6,0) circle (0.6 and 0.3) (0.8,2.3) circle (0.6 and 0.3);
      \draw[blue] (1.2,0) to[out=90,in=-90] (1.4,2.3) node[right,color=black] {$\gamma$} (0,0)node[draw=black,circle,inner sep=1pt,fill=black](A){} to[out=90,in=-90] (0.2,2.3);
      \node at (A) [left] {$\iota(0)$};
      \node[blue] at (0.7,1) {$u$};
      \node at (0,-1.2) {$\iota(B(a))$};
    \end{tikzpicture}
    \caption{The limiting Floer cylinder is asymptotic to the Hamiltonian orbit $\gamma$ at its left end. The right end traces out a loop in $M$ which passes through $\iota(0)$.}
    \label{fig:2}
  \end{figure}
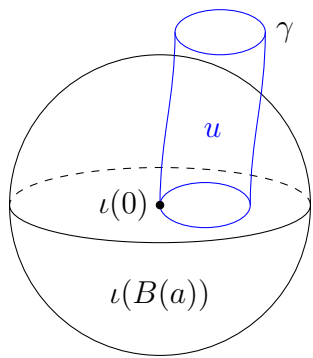

  Standard compactness arguments imply that the reparametrizations $u_{n}(s+\ell_{n},t)$ converge to a solution $u$ of \eqref{eq:compactness-up-to-breaking}, with $u(0,0)=\iota(0)$, after passing to a subsequence; see \S\ref{sec:length_diverging}. The energy bound from Lemma \ref{lemma:apriori_bound_kappa} implies that:
  \begin{equation*}
    \int \norm{\bd_{s}u}^{2}_{J}\d s \d t \le \max_{\Omega}H-\min_{M} H-\kappa(\mathfrak{B},\Omega).
  \end{equation*}
  The limit of $u(s,t)$ as $s\to -\infty$ is a 1-periodic orbit $\gamma$ of $X_{H}$. Since $t\mapsto u_{n}(s,t)$ lies in the free homotopy class of a $\zeta$-orbit, the Hamiltonian orbit $\gamma$ is non-contractible. Since $\d H=0$ on $\iota(B(a))$, it follows easily that $u(s,t)\not\in B(a)$ for $s$ sufficiently negative, and also $u$ is $J$-holomorphic on $u^{-1}(\iota(B(a)))$. The monotonicity theorem from \S\ref{sec:monotonicity_theorem} implies $a\le 2\int \norm{\bd_{s}u}^{2}_{J}\d s \d t,$ and hence:
  \begin{equation*}
    a\le 2(\max_{\Omega}H-\min_{M} H-\kappa(\mathfrak{B},\Omega))\le 2\epsilon-2\kappa(\mathfrak{B},\Omega).
  \end{equation*}
  Since $\epsilon$ was arbitrary, we conclude the desired result.
\end{proof}

\subsubsection{Sharpness of the estimate}
\label{sec:sharpness}

In this section we show that, for certain starshaped domains in $T^{*}M$, for $M=T^{n}$, the estimate in Theorem \ref{theorem:B} is sharp. Let
\begin{equation*}
  M:=M(a_{1},\dots,a_{n})=\R/a_{1}\Z\times \dots \times \R/a_{n}\Z,
\end{equation*}
where $0<a_{1}\le \dots\le a_{n}$. Let $q_{1},\dots,q_{n}$ be the corresponding coordinates (so $q_{i}$ is $\R/a_{i}\Z$-valued), and $p_{1},\dots,p_{n}$ the dual coordinates on $T^{*}M$. A natural star-shaped domain is given by $\Omega(r):=\set{p_{1}^{2}+\dots+p_{n}^{2}\le r^{2}}$.
\begin{prop}
  Let $\mathfrak{B}(\pi,t)=(q_{1}+a_{1}t,q_{2},\dots,q_{n})\in M$. Then the relative Gromov width of $(\Omega(r),M)$ equals $-2\kappa(\Omega(r),\mathfrak{B})=2a_{1}r$, and Theorem \ref{theorem:B} is sharp in this case.
\end{prop}
\begin{proof}
  Observe that $\Omega(r)$ contains $D^{n}(a_{1}/2)\times D^{n}(r)$ (the product of two Lagrangian disks). Then \cite[Lemma 3.2]{brocic_1} produces a symplectic embedding:
  \begin{equation*}
    B(2a_{1}r)\to D^{n}(a_{1}/2)\times D^{n}(r),
  \end{equation*}
  where $B(a)$ is the ball of symplectic capacity $a$ (i.e., has radius $(a/\pi)^{1/2}$). The symplectic embedding constructed maps $\R^{n}\cap B(a)$ onto the disk $D^{n}(a_{1}/2)\times \set{0}$. Thus the relative Gromov width is at least $2a_{1}r$.

  It remains to show that $-\kappa(\Omega(r),\mathfrak{B})=a_{1}r$. This follows from the calculations in the proof of Theorem \ref{theorem:exist_1}. Indeed, \eqref{eq:kappa_lambda} implies:
  \begin{equation*}
    \kappa(\Omega(r),\mathfrak{B})=\int_{0}^{1}\min_{\abs{p}^{2}\le r^{2}}\sum p_{i}q_{i}'(t)\,\d t=-a_{1}r,
  \end{equation*}
  using that $q'(t)=(a_{1},0,\dots,0)$. This completes the proof.
\end{proof}

\subsection{Outline of paper and discussion of results}
\label{sec:outline_and_discussion}

The rest of the paper is centered around proving Theorem \ref{theorem:main}.

In \S\ref{sec:floer_in_liouville}, we gather various facts about Floer's equation in Liouville manifolds (of which the cotangent bundle is a special case). A somewhat novel fact is the maximum principle we employ, \S\ref{sec:maximum_principle}. Our approach to the maximum principle is based on the arguments in \cite{merry_ulja}, but works for a larger class of complex structures. Unlike \cite{merry_ulja}, our approach does not rely on the Aleksandrov maximum principle, but assumes bounds on the first derivatives. In \S\ref{sec:apriori_first_derivative} we explain how to obtain these first derivative bounds using bubbling analysis.

The compactness statements are in \S\ref{sec:compactness}. In \S\ref{sec:transversality_gluing}, we prove various transversality, index theory, and gluing statements.

We should emphasize that, throughout, we work with a fairly large class of complex structures, namely those which are $\omega$-tame and translation invariant in the ends. The main results could be proved using only complex structures which are of SFT type in the ends (\`a la \cite{behwz}), but, in order to estimate the relative Gromov width in \S\ref{sec:relative_gromov_width}, one needs to allow arbitrary $\omega$-compatible complex structures near the zero section. This complicates the gluing analysis in \S\ref{sec:gluing} compared to, e.g., \cite[\S 4.6]{abbondandolo_schwarz_2}, as we cannot assume that $J_{p} TM_{p}=T^{*}M_{p}$ holds for $p\in M$.

\subsubsection{Further directions}
\label{sec:further_directions}

The upper bound on the relative Gromov width obtained in Theorem \ref{theorem:B} was based on the existence of a half-infinite Floer cylinder as in Figures \ref{fig:3} and \ref{fig:2}. Similar maps are considered in ``open-closed maps'' relating Hamiltonian Floer homology and the Lagrangian Floer homology of the zero-section.\footnote{The authors thank Egor Shelukhin for pointing this out.} It is plausible that the existence of such cylinders can be proved by considering how such open-closed maps act on homology groups. Since the Floer homology of Lagrangians involves pairs of nearby Lagrangians, some form of adiabatic degeneration is expected to be necessary, perhaps similar to the results in \cite{FloerW,cant_chen}. Such an approach goes beyond the scope of the paper, and is left for future study.

Another direction is to analyze the breakings in $\mathscr{M}$ as $\ell\to \infty$ more closely. One potential application of this analysis would be to emulate the construction of the Floer fundamental group of \cite{barraud_floer_fundamental_group,barraud_simone_fundamental_group}, and hopefully relate the fundamental group of the free loop space to Floer theoretic objects.

\subsection{Acknowledgements}
\label{sec:ack}

This research was carried out at Universit\'e de Montr\'eal. The authors wish to thank Octav Cornea and Egor Shelukhin for many fruitful discussions. The first author is supported by a ISM scholarship, and the second author is supported by funding from the CIRGET research group. Both authors are also supported by funding from the Foundation Courtois.

\section{Floer's equation in Liouville manifolds}
\label{sec:floer_in_liouville}

An \emph{exact symplectic manifold} is a smooth manifold $W$ equipped with a \emph{Liouville form} $\lambda$, i.e., $\d\lambda$ is a symplectic form. The associated \emph{Liouville vector field} $Z$ is defined by the equation $Z\intprod \d\lambda=\lambda$. If $\rho_{\sigma}$ is the time $\sigma$ flow by $Z$, then $\rho_{\sigma}^{*}\lambda=e^{\sigma}\lambda$.

A \emph{symplectization} is an exact symplectic manifold $W$ whose Liouville flow $Z$ is complete and defines a free and proper $\R$-action. In this case, there is an odd-dimensional quotient space $Y=W/\R$. The kernel of the Liouville form $\lambda$ is invariant under the flow and projects to a \emph{contact distribution} $\xi$ on $Y$. In this case we say that $W$ is the symplectization of $(Y,\xi)$, and typically denote it $W=SY$.

A \emph{Liouville manifold} is an exact symplectic manifold $W$ which contains a compact $Z$-invariant set $\mathrm{Sk}(W)$, called the \emph{skeleton}, so that:
\begin{enumerate}
\item the complement $\mathrm{Sk}(W)^{c}$ is the symplectization of a compact manifold $Y$, and
\item for any $z$, $\rho_{\sigma}(z)$ eventually enters any neighborhood of $\mathrm{Sk}(W)^{c}$ as $\sigma\to-\infty$,
\end{enumerate}
Essentially, one thinks of a Liouville manifold as a symplectization with its negative end ``filled.'' Liouville manifolds are studied in greater detail in \S\ref{sec:liouville_manifolds}.

\subsection{Symplectizations}
\label{sec:contact_prelims}
Let $W$ be a symplectization and let $Y=W/\R$, with $\xi$ being the projection of $\ker \lambda$. Note that $\xi$ is cooriented by requiring that $\lambda$ is positive on transverse vectors. The ``exact symplectomorphism type'' of $(W,\lambda)$ is completely determined by the (cooriented) contactomorphism type of $(Y,\xi)$:
\begin{prop}
  Let $(SY_{1},\lambda_{1}),(SY_{2},\lambda_{2})$ be symplectizations. Every contactomorphism $\varphi:Y_{1}\to Y_{2}$ preserving the coorientation is induced by a unique equivariant diffeomorphism $\Phi:SY_{1}\to SY_{2}$ satisfying $\Phi^{*}\lambda_{2}=\lambda_{1}$.
\end{prop}
\begin{proof}
  The result is well-known and left to the reader.
\end{proof}

\subsubsection{Contact Hamiltonians}
\label{sec:contact_hams}
A smooth function $H:SY\to \R$ is called a \emph{contact Hamiltonian} provided that one of the following equivalent conditions hold:
\begin{enumerate}
\item $Z\intprod \d H=H$,
\item $H\circ \rho_{\sigma}=e^{\sigma}H$,
\item $X_{H}\intprod \lambda=H$,
\end{enumerate}
These conditions imply that the flow $\Phi_{t}$ of $X_{H}$ commutes with the Liouville flow. In particular, $\Phi_{t}$ descends to a contact isotopy $\varphi_{t}:Y\to Y$. Every contact isotopy arises in this fashion.

If $Q$ is a non-vanishing contact Hamiltonian, then $h=H/Q$ is invariant under the Liouville flow, and hence $h$ induces a function $Y\to \R$. Other authors often call functions $h:Y\to \R$ ``contact Hamiltonians,'' and this is related to our notion via this correspondence $hQ=H$. The additional data of a non-vanishing contact Hamiltonian is equivalent to a choice of contact form via $\alpha=\lambda/Q$.

\subsubsection{Legendrians and Lagrangians}
\label{sec:legendrians_and_lagrangians}

A Lagrangian $L$ in $SY$ is called \emph{cylindrical} if $L$ is $Z$-invariant. In particular, $Z$ is tangent to $L$. It follows that $\lambda|_{L}=0$, and hence cylindrical Lagrangians are exact (in a strong sense).

Cylindrical Lagrangians are lifts of Legendrians $\Lambda\subset Y$, and we will denote this fact by writing $L=S\Lambda$.

\subsubsection{Translation invariant complex structures and metrics}
\label{sec:translation_invariant_complex_structures}

A complex structure $J$ on $SY$ is called \emph{translation invariant} provided $J\d\rho_{\sigma}=\d\rho_{\sigma}J$. We say that $J$ is $\omega$-\emph{tame} provided $\omega(-,J-)$ is a positive $2$-tensor, and say $J$ is $\omega$-\emph{compatible} if $\omega(-,J-)$ is also symmetric. In either case, we denote:
\begin{equation*}
  \norm{v}_{J}^{2}=\omega(v,Jv).
\end{equation*}
A Riemannian metric $\abs{-}$ is called \emph{translation invariant} provided that $\abs{\d\rho_{\sigma}v}=\abs{v}$. We note that $\norm{-}_{J}^{2}$ is not translation invariant. However, if $J$ is translation invariant and $Q$ is a positive contact Hamiltonian then $\abs{v}^{2}=Q(p)^{-1}\norm{v}_{J}^{2}$ is a translation invariant metric (where $v\in T_{p}W$).

\subsection{Liouville manifolds}
\label{sec:liouville_manifolds}

Let $(W,\lambda)$ be a Liouville manifold with skeleton $\mathrm{Sk}(W)$. Let $SY=\mathrm{Sk}(W)^{c}$. The contact manifold $Y$ is called the \emph{ideal boundary} of $W$.

\subsubsection{Star-shaped domains}
\label{sec:star-shaped domains}
Let $\Sigma\subset SY$ be the level set $\set{r=1}$ for a positive contact Hamiltonian $r$. Then:
\begin{equation*}
  \Omega:=\set{r\le 1}\cup \mathrm{Sk}(W)
\end{equation*}
is a compact codimension $0$ submanifold in $W$ (with boundary $\bd\Omega=\Sigma$). Domains of this form are referred to as \emph{starshaped domains}. By construction, there is a bijection between starshaped domains and positive contact Hamiltonians on $SY$.

\subsubsection{Admissible Hamiltonians and complex structures}
\label{sec:admissible_hamiltonians}
An \emph{admissible Hamiltonian for a starshaped domain $\Omega$} is a 1-periodic family of functions $H_{t}$ which are contact Hamiltonians outside $\Omega$, and so that every solution to:
\begin{equation}\label{eq:ham_ODE}
  \gamma'(t)=X_{H_{t}}(\gamma(t))
\end{equation}
passes through the interior of $\Omega$. If there is no risk of ambiguity, we will abbreviate the associated Hamiltonian vector field by $X_{t}=X_{H_{t}}$. It can be shown that admissibility is a generic condition, essentially because the $1$-periodic orbits at infinity should be isolated, but they always occur in a 1-dimensional family (by the equivariance with the Liouville flow).

We say that a 1-periodic family of complex structures $J_{t}$ is \emph{admissible for $\Omega$} provided $J_{t}$ is $\omega$-tame and $J_{t}$ is translation invariant on $\Omega^{c}$.

The key consequence of admissibility is the following estimate:
\begin{prop}\label{prop:Nsigma}
  Let $H_{t},J_{t}$ be an admissible Hamiltonian and complex structure for $\Omega$. For any $N>0$ there is $\sigma\ge 0$ so that, for $C^{1}$ loops $\gamma$,
  \begin{equation*}
    \gamma(t)\not\in \rho_{\sigma}(\Omega)\text{ for all $t$}\implies \mathfrak{n}(\gamma):=\int_{0}^{1}\norm{\gamma'(t)-X_{t}}^{2}_{J_{t}}\d t>N.
  \end{equation*}
\end{prop}
\begin{proof}
  See \cite[Lemma 3.6]{merry_ulja} for a similar result. If $\gamma(t)\not\in \Omega$ for all $t$, then compute $\mathfrak{n}(\rho_{s}(\gamma))=e^{s}\mathfrak{n}(\gamma)$ for $s\ge 0$. In particular, if the result holds for $\sigma,N$, then it holds for $\sigma+s,e^{s}N$. In particular, it suffices to prove the result for $N$ sufficiently small.

  Suppose, in search of a contradiction, that the statement does not hold for any $N$ with $\sigma=0$. Then we can find a sequence $\gamma_{n}$ remaining outside of $\Omega$ so that $\mathfrak{n}(\gamma_{n})\to 0$. The goal is to prove that $\gamma_{n}$ converges to a periodic orbit of $X_{t}$, contradicting the admissibility condition.

  We consider $\gamma(t)$ as living in $SY$, noting that $X_{t}$ extends to all of $SY$ by equivariance with Liouville flow.

  Fix a translation invariant metric $\abs{v}^{2}$, which is comparable to $Q^{-1}\norm{v}_{J_{t}}^{2}$ since $J_{t}$ is translation invariant. Since $Q$ is bounded from below on $\Omega^{c}$, conclude:
  \begin{equation*}
    \lim_{n\to\infty}\int_{0}^{1} \abs{\gamma_{n}'(t)-X_{t}(\gamma_{n})}^{2}\d t=0
  \end{equation*}
  Since $\abs{X_{t}}$ is uniformly bounded, we conclude that $\int \abs{\gamma_{n}'(t)}^{2}\d t$ is bounded. It follows that $\gamma_{n}$ has bounded diameter and is equicontinuous, and hence converges in $C^{0}$ to a $1$-periodic limit $\gamma_{\infty}$ (potentially after flowing by the Liouville flow and passing to a subsequence).

  Let $\mu(t)$ be the unique solution of $\mu'(t)=X_{t}(\mu(t))$ with $\mu(t_{0})=\gamma_{\infty}(t_{0})$ (we do not claim that $\mu$ is periodic). We claim that $\mu(t)=\gamma_{\infty}(t)$ on a neighborhood of $t_{0}$. Without loss, work in coordinate charts, set $t_{0}=0$, and estimate:
  \begin{equation*}
    \abs{\gamma_{n}(t)-\mu(t)}\le \epsilon_{n}+t^{1/2}[\int_{0}^{t}\abs{\gamma_{n}'(t)-X_{t}(\gamma_{n})}^{2}\d t]^{1/2}+C_{X}t\max_{0\le \tau\le t}\abs{\gamma_{n}(\tau)-\mu(\tau)},
  \end{equation*}
  where $\epsilon_{n}=\abs{\gamma_{n}(0)-\gamma_{\infty}(0)}$. Here we have used the fact that:
  \begin{equation*}
    \int_{0}^{t}\abs{X_{t}(\gamma_{n}(\tau))-X_{t}(\mu(\tau))}\d\tau\le C_{X}t\max_{0\le \tau\le t}\abs{\gamma_{n}(\tau)-\mu(\tau)},
  \end{equation*}
  so $C_{X}$ depends on the $C^{1}$ size of $X$. For $t$ small, we have $C_{X}t<1/2$. One obtains:
  \begin{equation*}
    \max_{0\le \tau\le t}\abs{\gamma_{n}(\tau)-\mu(\tau)}\le 2\epsilon_{n}+2t^{1/2}[\int_{0}^{t}\abs{\gamma_{n}'(t)-X_{t}(\gamma_{n})}^{2}\d t]^{1/2},
  \end{equation*}
  and taking the limit $n\to\infty$ one concludes, for $t$ small, that $\gamma_{\infty}(t)=\mu(t)$. It follows that $\gamma_{\infty}(t)$ is an integral curve for $X_{t}$, and hence $X_{t}$ has a periodic orbit outside of $\Omega$. This contradicts admissibility, and completes the proof.
\end{proof}

\subsubsection{Hamiltonian action functional}
\label{sec:hamiltonian_action}
Given an admissible Hamiltonian $H_{t}$, the associated action functional is defined for $\gamma\in C^{1}(\R/\Z,W)$ and is given by:
\begin{equation*}
  \mathscr{A}(\gamma)=\int_{0}^{1}H_{t}(\gamma(t))\d t-\gamma^{*}\lambda.
\end{equation*}
The critical points of $\mathscr{A}$ are precisely the solutions to \eqref{eq:ham_ODE}.

\subsubsection{Floer cylinders}
\label{sec:floer_cylinders}

Given data $H_{t},J_{t}$ admissible for $\Omega$ the \emph{Floer equation} is:
\begin{equation}
  \label{eq:floer_cylinder}
  \left\{
    \begin{aligned}
      &u:[a,b]\times \R/\Z\to W\\
      &\bd_{s}u+J_{t}(u)(\bd_{t}u-X_{t}(u))=0,
    \end{aligned}
  \right.
\end{equation}
allowing $a=-\infty$ or $b=+\infty$. Abbreviate $u(s)=u(s,-)$, so that $u(s)$ is considered as a loop. The Floer equation is interpreted as the negative gradient flow for the symplectic action:
\begin{prop}\label{prop:apriori_energy}
  Let $u$ solve \eqref{eq:floer_cylinder}, and let $\mathscr{A}$ be the action functional for $H_{t}$. Then:
  \begin{equation*}
    E(u)=\int \norm{\bd_{s}u}_{J_{t}}^{2}\d s\d t =    \int \norm{\bd_{t}u-X_{t}(u)}_{J_{t}}^{2}\d s\d t = \mathscr{A}(u(a))-\mathscr{A}(u(b)).
  \end{equation*}
\end{prop}
\begin{proof}
  The computation is standard, and is left to the reader.
\end{proof}

\subsubsection{Maximum principle}
\label{sec:maximum_principle}

The next result is a crucial ingredient in proving various moduli spaces are compact.

\begin{theorem}\label{theorem:maximum_principle}
  Let $u_{n}:[a_{n},b_{n}]\times \R/\Z\to W$ be a sequence of Floer cylinders so that:
  \begin{enumerate}
  \item $H_{t},J_{t}$ are admissble for $\Omega$,
  \item $b_{n}-a_{n}$ remains bounded from below by a positive number,
  \item $E(u_{n})$ remains bounded,
  \item $\abs{\bd_{s}u_{n}}$, $\abs{\bd_{t}u_{n}}$ remain bounded (for a metric which is translation invariant outside of $\Omega$)
  \end{enumerate}
  Then there is $\sigma\ge 0$ so that $u_{n}([a_{n},b_{n}]\times \R/\Z)\subset \rho_{\sigma}(\Omega)$ for all $n$.
\end{theorem}
\begin{proof}
  See \cite[Theorem 3.10]{merry_ulja} for a similar result. Without loss of generality, suppose that $a_{n}=0$ and $b_{n}=\ell_{n}$. The case when $a_{n}=-\infty$ or $b_{n}=\infty$ is similar. Suppose that $\ell_{n}\ge \ell$, $E(u_{n})\le E$, and $\abs{\bd_{s}u_{n}},\abs{\bd_{t}u_{n}}$ remain bounded by $C$.

  Proposition \ref{prop:Nsigma} with $N=2E/\ell$ guarantees some $\sigma_{0}$ so that
  \begin{equation*}
    \gamma(t)\not\in \rho_{\sigma_{0}}(\Omega)\text{ for all $t$}\implies \int_{0}^{1}\norm{\gamma'-X_{t}(\gamma)}_{J_{t}}^{2} \d t>2E/\ell.
  \end{equation*}
  Since
  \begin{equation*}
    E(u_{n})=\int_{0}^{\ell_{n}}\int_{0}^{1}\norm{\bd_{t}u_{n}-X_{t}(u_{n})}_{J_{t}}^{2} \d t\le E,
  \end{equation*}
  we conclude that, for every $s$, there is $s'$ so $\abs{s-s'}\le \ell/2$ so that $u_{n}(s',t')\in \rho_{\sigma_{0}}(\Omega)$ for some $t'$. But then, since the first derivative is bounded by $C$, we conclude that, for any $t$,
  \begin{equation*}
    \mathrm{dist}(u_{n}(s,t),\rho_{\sigma_{0}}(\Omega))\le C\abs{t-t'}+C\abs{s-s'}\le C(1+\ell/2),
  \end{equation*}
  Clearly we can pick $\sigma$ large enough so that the $C(1+\ell/2)$ ball (as measured by the translation invariant metric) around $\rho_{\sigma_{0}}(\Omega)$ is contained in $\rho_{\sigma}(\Omega)$. It follows that $u(s,t)\in \rho_{\sigma}(\Omega)$ for all $s,t$. Clearly $\sigma$ depends only on $\ell,C,E$. This completes the proof.
\end{proof}

The requirement that the first derivatives of $u_{n}$ are bounded is often achievable using bubbling analysis, see \S\ref{sec:apriori_first_derivative}.

\subsubsection{The monotonicity theorem for Darboux balls}
\label{sec:monotonicity_theorem}
Let $B(a)\subset \C^{n}$ be the open ball of capacity $a$, i.e., if $\Pi$ is a complex line, then $\Pi\cap B(a)$ has symplectic area $a$. The monotonicity theorem asserts that every holomorphic curve which ``properly'' passes through the center of the ball has symplectic area at least $a$.
\begin{prop}
  Let $(\Sigma,\bd\Sigma,j)$ be a compact Riemann surface and $u:\Sigma\to \C^{n}$ be a holomorphic curve (for the standard complex structure) so that $u(\bd\Sigma)\cap B(a)=\emptyset$ and $u^{-1}(0)\ne \emptyset$. Then the symplectic area of $u$ is at least $a$.
\end{prop}
\begin{proof}
  We sketch the argument. See \cite[\S5.2]{wenc} for more details.

  It suffices to prove the theorem when $a=1$. Let $Q(z)=\pi \abs{z}^{2}$, so $Q<1$ cuts out $B(1)$. Interpret $SY=\C^{n}-\set{0}$ as the symplectization of the standard contact sphere $Y$, so that $Q$ is a contact Hamiltonian.

  Analytic continuation implies that $u^{-1}(0)$ is a finite (non-empty) set $\Gamma$. Interpret $u$ as a \emph{punctured} holomorphic map $\Sigma-\Gamma\to SY$, so that $\bd \Sigma$ is mapped into in the region $Q\ge 1$. It can be shown that the standard complex structure is of SFT type (for the standard Reeb field $R=X_{Q}$), in the sense of \cite{behwz}, and thus $u$ will be asymptotic at its punctures to Reeb orbits. The contact form is $\alpha=\lambda/Q$. Standard results from SFT then imply that the integral of $\d\alpha$ will be at least the minimal positive action of a Reeb orbit. It is not hard to see that:
  \begin{equation*}
    \text{minimal positive action}\le\int_{\Sigma-\Gamma} \d\alpha =\int_{\bd\Sigma} \alpha \le \int_{\bd\Sigma} \lambda =\int_{\Sigma} \d\lambda=\text{symplectic area}.
  \end{equation*}
  It is also a standard fact that, with our conventions, the minimal positive action of a Reeb orbit is exactly $1$. Thus we conclude the desired result.
\end{proof}
\begin{cor}
  Let $(\Sigma,\bd\Sigma)$ be a compact Riemann surface, $\bd\Sigma=\bd\Sigma_{1}\sqcup \bd\Sigma_{2},$ and suppose $u:\Sigma\to \C^{n}$ is holomorphic with $u(\bd\Sigma_{1})\cap B(a)=\emptyset$, $u(\bd\Sigma_{2})\subset \R^{n}$, and $u^{-1}(0)\ne\emptyset$. Then the symplectic area of $u$ is at least $a/2$.
\end{cor}
\begin{proof}
  Use the Schwarz reflection principle and apply the previous result, bearing in mind we are using the standard complex structure.
\end{proof}

\section{A priori energy bounds and compactness of the moduli space}
\label{sec:compactness}

\subsection{Energy bounds imply bounds on all derivatives}
\label{sec:energy_bounds}

The results in this section explain why, for Floer's equation on Liouville manifolds, energy bounds are sufficient to obtain strong compactness results. In general, energy bounds are \emph{not} sufficient because of the bubbling phenomenon, but bubbling can be precluded in Liouville manifolds.

\subsubsection{Elliptic regularity}
\label{sec:elliptic_regularity}

Let $W$ be an almost complex manifold. The elliptic regularity result in this section show that $C^{1}$ bounds are enough to ensure convergence results for sequences of solutions to holomorphic type equations. The result is stated for maps defined on either $\Sigma(1)=D(1)$, in which case $\bd\Sigma(1)=\emptyset$, or $\Sigma(1)=D(1)\cap \cl{\mathbb{H}}$, in which case $\bd\Sigma(1)=D(1)\cap \R$.

\begin{prop}
  Suppose that $J_{n,z}\to J_{\infty,z}$, $A_{n,z}\to A_{\infty,z}$, $L_{n,z}\to L_{\infty,z}$ are $C^{\infty}$ convergent sequences of complex structures, vector fields, and totally real submanifolds, respectively, and $u_{n}:\Sigma(1)\to W$ satisfies:
  \begin{equation*}
    \left\{
      \begin{aligned}
        &\bd_{s}u_{n}+J_{n,z}(u_{n})\bd_{t}u_{n}=A_{n,z}(u_{n})\\
        &u_{n}(z)\in L_{n,z}\text{ for }z\in \bd\Sigma(1)
      \end{aligned}
    \right.
  \end{equation*}
  Suppose $u_{n}(0)$ converges and $\abs{\bd_{s} u_{n}}, \abs{\bd_{t}u_{n}}$ remain bounded, as measured by a complete metric on the target $W$. Then there is a smooth map $u_{\infty}$ satisfying:
  \begin{equation*}\left\{
    \begin{aligned}
      &\bd_{s}u_{\infty}+J_{\infty,z}(u_{\infty})\bd_{t}u_{\infty}=A_{\infty,z}(u_{\infty})\\
      &u_{\infty}(z)\in L_{\infty,z}\text{ for }z\in \bd\Sigma(1)
    \end{aligned}\right.
  \end{equation*}
  so that $u_{n}$ has a subsequence converging to $u_{\infty}$ in $C^{\infty}(\Sigma(r),W)$ for every $r<1$.
\end{prop}
\begin{proof}
  The result is a well-known bootstrapping argument. One reduction is that it suffices to prove the result for $r$ sufficiently small (and then apply the result to a covering by small disks). See \cite[Proof of Theorem 2.24]{wendl-sft}. Another reduction is that one can replace $u_{n}(z)$ by $\Phi_{n,z}(u_{n}(z))$ so that $\Phi_{n,z}$ converges and $\Phi_{n,z}(L_{n,z})=L$ for a fixed totally real submanifold $L$. In coordinates near the limit of $u_{n}(0)$, one can arrange that $L$ is $\R^{n}\subset \C^{n}$. Moreover, one can apply a $z$-dependent coordinate change so that $J|_{\R^{n}}=J_{0}|_{\R^{n}}$ is the standard complex structure. For a final reduction, replace $u_{n}(s,t)$ by $u_{n}(s,t)+a_{n}(s)+tb_{n}(s)$ where $a_{n},b_{n}$ are $\R^{n}$-valued. Then we can arrange so that $\bd_{s}u_{n}+J_{0}\bd_{t}u_{n}=0$ holds along $\bd \Sigma(r)$. The $C^{\infty}$ size of $a_{n},b_{n}$ (and the other coordinate changes) are controlled by the $C^{\infty}$ sizes of $A_{n,z}$, $J_{n,z}$, $L_{n,z}$.

  Given an equation of the form $\bd_{s}u+J_{z}(u)\bd_{t}u_{n}=A_{z}(u)$, one can apply $\bd_{s}-J_{z}(u)\bd_{t}$ to both sides to obtain an equation involving the standard Laplacian $\Delta=\bd_{s}^{2}+\bd_{t}^{2}$.

  One then shows that $u_{n}$ is bounded in $W^{k,2}(\Sigma(r))$ for every $k$ by using the elliptic estimates for the  Laplacian with Dirichlet and Neumann boundary conditions (applied to the real and imaginary parts), see \cite[Appendix C]{robbinsalamon} and \cite{cant_chen}. Then one uses the Sobolev embedding theorem to prove that $u_{n}$ is bounded in $C^{k}(\Sigma(r))$ for all $k$. A standard extension of the Arzel\`a-Ascoli theorem ensures that $C^{k}$ bounds for all $k$ ensures subsequences which converge in $C^{\infty}$. The details are left to the reader.
\end{proof}

\subsubsection{A priori bounds on first derivatives}
\label{sec:apriori_first_derivative}
Bubbling analysis for holomorphic curve type equations allows one to conclude that either a sequence $u_{n}$ has a bounded first derivative, or there exists a non-constant holomorphic plane $v:\C\to W$ or half-plane with bounded first derivative. One obtains $v$ by rescaling shrinking regions on the domain of $u_{n}$, and then applying the elliptic regularity results from \S\ref{sec:elliptic_regularity}. The energy of $v$ will be bounded from above by the limiting energy of $u_{n}$, using the conformal invariance of energy. For details on bubbling analysis see \cite[Appendix C]{cant_chen}.

In Liouville manifolds and symplectizations, there are no finite energy holomorphic planes or half planes, and hence $C^{1}$ bounds follow from energy bounds. The precise non-existence statement is:
\begin{prop}
  Let $\Sigma=\C$ or $\Sigma=\cl{\mathbb{H}}$. Let $(W,\lambda)$ be an exact symplectic manifold, let $L$ be an exact Lagrangian, and let $J$ be an $\omega$-tame complex structure.  There is no non-constant holomorphic map $u:\Sigma\to (W,L)$ whose image lies in a compact set. In addition, if $W$ is a Liouville manifold (or positive half of a symplectization), $L$ is cylindrical outside a compact set, and $J$ is translation invariant outside a compact set, then there are no non-constant holomorphic maps $u:\Sigma\to (W,L)$ with finite symplectic area.
\end{prop}
\begin{proof}
  We argue by contradiction. The first part is proved via an elementary differential inequality. Let $\Sigma(r)=\Sigma\cap D(r)$, and define:
  \begin{equation*}
    A(r)=\int_{\Sigma(r)}\norm{\bd_{s}u}^{2}_{J_{s}}\d s\d t=\int_{\Sigma(r)}u^{*}\d\lambda.
  \end{equation*}
  Since $u$ is non-constant, $A(r)$ is positive. Using Stokes' theorem, estimate:
  \begin{equation*}
    A(r)^{2}\le C rA'(r)\implies \frac{1}{C}\log(r_{2}/r_{1})\le \frac{1}{A(r_{1})}-\frac{1}{A(r_{2})},
  \end{equation*}
  where $r_{1}<r_{2}$ and $C$ depends on the (bounded) size of $\lambda$ on the image of $u$ (which we assume is precompact). Take the limit as $r_{2}\to\infty$ to conclude a contradiction.

  For the second part, we appeal to the following fact: for every $\delta>0$, there is an $\epsilon>0$ so that, if a holomorphic cylinder/strip $u:[-R-1,R+1]\times S\to W$, where $S=\R/\Z$ or $S=[0,1]$, satisfies $\int \abs{\bd_{s}u}^{2}\d s\d t<\epsilon$, then the diameter of $u([-R,R]\times S)$ is less than $\delta$ (using some metric which is translation invariant in the ends). In the case when $S=[0,1]$ we require that $u$ takes boundary values on $L$. This fact can be proved using the mean-value property for the energy density $\abs{\bd_{s}u}^{2}$, see \cite{mcduffsalamon,robbinsalamon}. A similar argument is given in \cite[Proposition 6.1]{cant_chen}.

  Since $u$ has finite symplectic area, and $\abs{-}=Q^{-1}\norm{-}_{J}$ defines a translation invariant metric, we conclude that regions of the form $\Sigma\setminus \Sigma(r)$ have bounded diameter for $r$ sufficiently large. Since the translation invariant metric is complete, this implies that $u$ has a precompact image, and hence we can apply the first part. This completes the proof.
\end{proof}

\subsection{A priori energy estimate for sufficiently negative Hamiltonians}
\label{sec:suff_neg}

\begin{prop}\label{prop:suff_neg_sufficient}
  Suppose that $H_{t}$ is sufficiently negative for $\Omega,\mathfrak{B}$, and let $\gamma(t)$ be a lift of $\mathfrak{B}(\pi,t)$ for some $\pi$. Then:
  \begin{equation}\label{eq:apriori_estimate}
    \mathscr{A}(\gamma)\le \sup\set{\mathscr{A}(\eta):\eta(t)\in \Omega\text{ for all $t$, and $\eta$ is a lift of $\mathfrak{B}$}}<\infty.
  \end{equation}
\end{prop}
\begin{proof}
  Let $r$ be a contact Hamiltonian so that $r(\bd\Omega)=1$. If $\gamma(t)\in \Omega$ for all $t$, then we are done. Otherwise, there is some interval $[t_{0},t_{1}]\subset \R/\Z$ so that $r(\gamma(t))>1$ for all $t\in [t_{0},t_{1}]$. Consider the contribution to the action coming from this arc:
  \begin{equation*}
    a=\int_{t_{0}}^{t_{1}}H_{t}(\gamma(t))\d t-\gamma^{*}\lambda.
  \end{equation*}
  We claim that $a\le 0$. To see why, pick some loop $\xi(t)$ lifting $\mathfrak{B}(\pi,t)$ so $\xi(t)=\gamma(t)$ for $t\in [t_{0},t_{1}]$ and $r(\xi(t))>1$ for all $t$. Then, if $a>0$, we can use positive functions $\beta(t)$ supported in $(t_{0},t_{1})$ to make $\mathscr{A}(\rho_{\beta(t)}\xi(t))$ arbitrarily large, using the identity:
  \begin{equation}\label{eq:identity_rescaling}
    \mathscr{A}(\rho_{\beta(t)}\xi(t))=\int_{0}^{1}e^{\beta(t)}(H_{t}(\xi(t))\d t-\xi^{*}\lambda).
  \end{equation}
  This contradicts the hypothesis that $H_{t}$ was sufficiently negative for $\mathfrak{B}$.

  Since $t_{0},t_{1}$ were arbitrary, we conclude that $H_{t}(\gamma(t))\d t-\gamma^{*}\lambda\le 0$ holds pointwise whenever $r(\gamma(t))>1$ (using standard orientation of $\R/\Z$). Pick $\beta(t)=-\log(r(\gamma(t)))$ for $t\in [t_{0},t_{1}]$ and suppose $\beta(t)\le 0$ is supported in region where $r(\gamma(t))>1$. Then:
  \begin{equation*}
    \mathscr{A}(\gamma(t))\le \mathscr{A}(\rho_{\beta(t)}\gamma(t)),
  \end{equation*}
  using \eqref{eq:identity_rescaling}. We have achieved $\rho_{\beta(t)}\gamma(t)\in \Omega$. Repeating this argument shows that $\mathscr{A}(\gamma)$ is less than the action of a loop which arbitrarily close to $\Omega$. Taking a limit completes the proof.
\end{proof}

Combining Propositions \ref{prop:suff_neg_sufficient} and \ref{prop:apriori_energy}, one concludes that, if $H_{t}$ is sufficiently negative for $\Omega,\mathfrak{B}$, then, for any $(u,\pi,\ell)\in \mathscr{M}(H_{t},J_{t},\mathfrak{B})$,
\begin{equation*}
  \begin{aligned}
    E(u)=\mathscr{A}(u(0))-\mathscr{A}(u(\ell))&\le \sup\set{\mathscr{A}(\eta):\eta(t)\in \Omega\text{ and }\eta\text{ lifts $\mathfrak{B}$}}-\int_{0}^{1}\min_{M} H_{t}\,\d t.
  \end{aligned}
\end{equation*}
and a straightforward estimate for $\mathscr{A}(\eta)$ completes the proof of Lemma \ref{lemma:apriori_bound_kappa}.

\subsection{Compactness of the moduli space}
\label{sec:compactness_of_moduli}

Let $\mathscr{M}=\mathscr{M}(\mathfrak{B},H_{t},J_{t})$ be the moduli space from \S\ref{sec:definition_of_moduli_space}, and suppose that $H_{t}$ is sufficiently negative for $\mathfrak{B}$.

The results of \S\ref{sec:suff_neg} imply that any sequence $(u_{n},\pi_{n},\ell_{n})\in \mathscr{M}$ will satisfy an a priori energy bound $E(u_{n})\le E(\mathscr{M})<\infty$. If $\ell_{n}$ remains bounded from below, there is enough room in the domain to perform bubbling analysis, and we conclude from \S\ref{sec:apriori_first_derivative} that the derivatives of $u_{n}$ will be bounded. The maximum principle \S\ref{sec:maximum_principle} applies and we conclude $u_{n}$ satsfies $C^{0}$ bounds.

If $\ell_{n}$ also remains bounded from above, then standard applications of Arzel\`a-Ascoli and elliptic regularity imply that $(u_{n},\pi_{n},\ell_{n})$ will converge, after passing to a subsequence. In other words, the map $\ell:\mathscr{M}\to (0,\infty)$ is proper. This completes the proof of the properness part of Theorem \ref{theorem:main}.

It remains to analyze the case when $\ell_{n}\to 0$ and when $\ell_{n}\to \infty$.

\subsubsection{Length shrinking to zero}
\label{sec:length_shrinking_to_zero}
Assume $H_{t}$ is sufficiently negative for $\mathfrak{B}$, to ensure the a priori energy bound. If $(u_{n},\pi_{n},\ell_{n})$ is a sequence in $\mathscr{M}$ and $\ell_{n}\to 0$, then the domain $[0,\ell_{n}]\times \R/\Z$ degenerates and there may not be enough space to apply bubbling analysis. However, one can apply the conformal reparametrization:
\begin{equation*}
  v_{n}:[0,1]\times \R/\ell_{n}^{-1}\Z\hspace{1cm}v_{n}(s,t)=u_{n}(\ell_{n}s,\ell_{n}t),
\end{equation*}
to obtain a new sequence with better control on the domain. The Floer equation for $u_{n}$ implies the following equation for $v_{n}$:
\begin{equation*}
  \left\{
    \begin{aligned}
      &\bd_{s}v_{n}+J_{\ell_{n}t}(v_{n})(\bd_{t}v_{n}-\ell_{n}X_{\ell_{n}t}(v_{n}))=0,\\
      &\pr\circ v_{n}(0,t)=\mathfrak{B}(\pi_{n},\ell_{n}t),\\
      &v_{n}(1,t)\in M.
    \end{aligned}
  \right.
\end{equation*}
The bubbling analysis from \S\ref{sec:apriori_first_derivative} can be applied and one concludes that $v_{n}$ has a bounded first derivative. In particular $v_{n}$ remains a bounded distance from the zero section (using the fact that its right end lies on the zero section).

If $\abs{\bd_{s}v_{n}}$ does not converge to zero, there is $s_{n},t_{n},\epsilon$ so that $\abs{\bd_{s}v_{n}(s_{n},t_{n})}>\epsilon$ for some subsequence. Pass to a further subsequence to ensure that $\ell_{n}t_{n}$ and $\pi_{n}$ converge to some $t_{\infty}\in \R/\Z$, $\pi_{\infty}\in P$, respectively. By the elliptic regularity result in \S\ref{sec:elliptic_regularity}, the rotated solutions:
\begin{equation*}
  w_{n}(s,t)=v_{n}(s,t+t_{n}),
\end{equation*}
have a subsequence which converges on compact sets to a \emph{non-constant} solution $w_{\infty}$ of the problem:
\begin{equation*}
  \left\{
    \begin{aligned}
      &w_{\infty}:[0,1]\times \R\to W,\\
      &\bd_{s}w_{\infty}+J_{t_{\infty}}(w_{\infty})\bd_{t}w_{\infty}=0,\\
      &\pr\circ w_{\infty}(0,t)=\mathfrak{B}(\pi_{\infty},t_{\infty}),\\
      &w_{\infty}(1,t)\in M.
    \end{aligned}
  \right.
\end{equation*}
Since $w_{\infty}$ is holomorphic and has boundary on Lagrangians on which the primitive $\lambda$ vanishes, $w_{\infty}$ must have zero energy and hence must be constant. This contradicts our assumption, and hence the derivatives $\abs{\bd_{s}v_{n}}$ must uniformly converge to zero.

This line of reasoning implies:
\begin{prop}\label{prop:small_ell}
  For $\epsilon>0$, there is $\ell_{0}=\ell_{0}(H_{t},J_{t},\mathfrak{B})$ so that every $(u,\pi,\ell)\in \mathscr{M}$ with $\ell\le \ell_{0}$ satisfies $\mathrm{dist}(u(s,t),\mathfrak{B}(\pi,t))<\epsilon.$
\end{prop}
\begin{proof}
  Suppose not. Then there is a sequence $u_{n},\pi_{n},\ell_{n}$ with $\ell_{n}\to 0$ but which remains $\epsilon$ far from $\mathfrak{B}(\pi_{n},t)$. The conformal reparametrization $v_{n}(s,t)=u_{n}(\ell_{n}s,\ell_{n}t)$ then remains $\epsilon$ far from $\mathfrak{B}(\pi_{n},t\ell_{n})$ (in the uniform distance). As shown above, the derivatives of $v_{n}$ converge to zero. Since the left end of $v_{n}$ is a lift of $\mathfrak{B}(\pi_{n},t\ell_{n})$ and the right end lies on the zero section, the distance between $v_{n}(s,t)$ and $\mathfrak{B}(\pi_{n},t\ell_{n})$ must eventually be uniformly less than $\epsilon$, as desired.
\end{proof}

\subsubsection{Length diverging to infinity}
\label{sec:length_diverging}

Consider a sequence $(u_{n},\pi_{n},\ell_{n})$ with $\ell_{n}\to\infty$ and the reparametrizations $w_{n}(s,t)=v_{n}(s+\ell_{n},t)$ defined on $(-\ell_{n},0]\times\R/\Z$. The gradient bound and consequent maximum principle imply that $w_{n}$ converges on compact sets to a limit map $w:(-\infty,0]\times\R/\Z\to W$. The elliptic regularity results from \S\ref{sec:elliptic_regularity} imply the convergence is in $C^{\infty}$ on compact sets, the limit $w$ is smooth and satisfies:
\begin{equation*}
  \left\{
  \begin{aligned}
    &w:(-\infty,0]\times \R/\Z\to T^{*}M,\\
    &\bd_{s}w+J_{t}(w)(\bd_{t}w-X_{H_{t}}(w))=0,\\
    &w(0,t)\in M.
  \end{aligned}
  \right.
\end{equation*}
Convergence on compact sets implies:
\begin{equation*}
  \int\norm{\bd_{s}w}_{J_{t}}^{2}\d s\d t \le \limsup_{n\to\infty}\int\norm{\bd_{s}u_{n}}_{J_{t}}^{2}\d s\d t.
\end{equation*}
Since $\norm{\bd_{s}w}_{J_{t}}^{2}=\norm{\bd_{t}w-X_{H_{t}}(w)}_{J_{t}}^{2}$, a standard argument in Floer theory implies that, for every sequence $s_{n}\to-\infty$, there is a subsequence so $w(s_{n},-)$ converges to a 1-periodic orbit for $X_{H_{t}}$. This is sufficient for the results in \S\ref{sec:existence_orbits} and \S\ref{sec:relative_gromov_width}.

\section{Transversality, index theory, and gluing}
\label{sec:transversality_gluing}
The goals of this section are to prove the moduli space $\mathscr{M}$ is a manifold whose dimension equals $\dim P+1$ (transversality and index theory), and to give an explicit description of the moduli space when the length parameter $\ell$ tends to zero (gluing).

\subsection{Transversality}
\label{sec:transversality}
The arguments in this section are standard. Experts can safely skip to \S\ref{sec:index_theory}.

\subsubsection{Injectivity results for Floer cylinders}
\label{sec:injectivity_results}

It is well-known that moduli spaces of holomorphic maps will be transverse for generic domain independent data near sufficiently injective maps. At the other extreme, one can achieve transversality without any injectivity assumptions if one allows domain dependent data. The Floer equation \eqref{eq:floer_cylinder} is only partially domain dependent, as the data is $s$-independent.

Fix $u:[a,b]\times \R/\Z\to W$ solving Floer's equation \eqref{eq:floer_cylinder} with $-\infty<a<b<\infty$. Following \cite{floer_hofer_salamon_transversality}, let us say that a point $(s_{0},t_{0})$ is \emph{regular} provided that (i) $\bd_{s}u(s_{0},t_{0})\ne 0$, and (ii) $u(s,t_{0})=u(s_{0},t_{0})$ if and only if $s=s_{0}$. Roughly speaking, regular points are ``injective points'' for the slices $t=t_{0}$.

\begin{prop}
  Suppose that $\bd_{s}u$ is not identically zero. Then the set of regular points is open and dense.
\end{prop}

\begin{proof}
  The argument follows \cite[Theorem 4.3]{floer_hofer_salamon_transversality}. That the regular points are open is straightforward to show, and is left to the reader.

  One shows $\xi=\bd_{s}u$ solves an equation of the form: $$\bd_{s}\xi+J(z)\bd_{t}\xi+A(z)\xi=0.$$ Then the similarity principle, see \cite[Theorem 2.2]{floer_hofer_salamon_transversality} or \cite[Theorem 2.32]{wendl-sft}, implies that $\xi=\bd_{s}u$ is either identically zero or has isolated zeroes. Since we assume $\bd_{s}u$ is not identically zero, the set of points where (i) holds is open and dense.

  Next, suppose that $z_{0}=(s_{0},t_{0})$ satisfies (i) and $s_{0}\in (a,b)$, but $z_{0}$ cannot be approximated by points which satisfy (ii). Let $P_{-}(z_{0})=(a,t_{0})$ and $P_{+}(z_{0})=(b,t_{0})$. By slightly shifting $z_{0}$, we may suppose that the $\delta$-square around $z_{0}$, denoted $Q_{\delta}(z_{0})$, satisfies $u(Q_{\delta}(z_{0}))\cap u(Q_{\delta}(P_{\pm}(z_{0})))=\emptyset$ for $\delta$ sufficiently small (this uses that $\bd_{s}u(z_{0})\ne 0$). Roughly speaking, $u(z_{0})$ remains away from the pieces of the left and right boundary circles with the same $t$ coordinate.

  By the argument in the first paragraph, there are only finitely many points $z$ in the rectangle $[a+\delta,b-\delta]\times [t_{0}-\delta,t_{0}+\delta]$ with $\bd_{s}u(z)=0$. Abbreviate $U(z)=(t(z),u(z))$. Shifting $z_{0}$ and shrinking $\delta$, we may further suppose that:
  \begin{equation*}
    z_{1}\in Q_{\delta}(z_{0})\text{ and }U(z_{2})=U(z_{1})\implies \bd_{s}u(z_{2})\ne 0.
  \end{equation*}
  By the covering argument given on \cite[pp.\ 262]{floer_hofer_salamon_transversality}, one can fix $\epsilon$ and construct: $$\phi:Q_{\delta}(z_{0})\to (a,b)\times (t_{0}-\delta,t_{0}+\delta)$$ so that $\phi(z)\ne z$, $U(\phi(z))=U(z)$, and so the diameter of $\phi$ is smaller than $\epsilon$ (after shifting $z_{0}$ slightly and shrinking $\delta$ small enough). Morally, the fact that we cannot make $U$ injective near $z_{0}$ is used to construct $\phi$.

  Then $u(\phi(z))=u(z)$ for $\phi(s,t)=(f(s,t),t)$. Since the $s$-derivative of $u$ is non-zero on the image of $\phi$, and $\phi$ has a small diameter, $f(s,t)$ is smooth (appealing to the constant rank version of the inverse function theorem). Invoking the Floer equation yields:
  \begin{equation*}
    \bd_{s}u(\phi)\cdot \bd_{s}f+J(u(\phi))\bd_{s}u(\phi)\cdot \bd_{t}f= -J(u(\phi))(\bd_{t}u(\phi)-X_{t}(u(\phi))).
  \end{equation*}
  Since $\bd_{s}u$ and $J(u)\bd_{s}u$ are linearly independent, there is a unique solution $\bd_{s}f=1$ and $\bd_{t}f=0$. Thus we conclude $u(s+s',t)=u(s,t)$ for some $s'\ne 0$. This holds for $s,t\in Q_{\delta}(z_{0})$, but it extends to the entire circle $\set{s_{0}}\times \R/\Z$ by unique continuation.

  Finally, the fact that the actions of the two loops $u(s_{0},-)$ and $u(s_{0}+s',-)$ are the same implies that the energy of $u$ restricted to the region $[s_{0},s_{0}+s']\times \R/\Z$ vanishes. This contradicts the fact that $\bd_{s}u$ has isolated zeroes, and this completes the proof.
\end{proof}

\subsubsection{Linearization framework and the Sard-Smale theorem}
\label{sec:linearization_framework}

In order to prove the moduli space has the structure of a finite dimensional manifold, one embeds it into an appropriate Banach manifold and applies the implicit function theorem.

As in \S\ref{sec:definition_of_moduli_space}, fix a smooth family $\mathfrak{B}(\pi,t)$ of loops in $M$, $\pi\in P$, $t\in \R/\Z$, and let $W=T^{*}M$. Fix some $p>2$.

Let $\mathscr{W}^{1,p}$ be the Banach manifold of triples $(u,\pi,\ell)$ where $\ell>0$,
\begin{equation*}
  \left\{
    \begin{aligned}
      &u:[0,\ell]\times\R/\Z\to W\text{ is of class }W^{1,p},\\
      &\pr\circ u(0,t)=\mathfrak{B}(\pi,t),\\
      &u(\ell,t)\in M.
    \end{aligned}
  \right.
\end{equation*}
One thinks of this as a total space of a bundle $\mathscr{W}^{1,p}\to P\times (0,\infty)$. It is straightforward exercise in nonlinear analysis to locally trivialize this bundle and identify it with an open subset of a Banach space. The Banach manifold $\mathscr{W}^{1,p}$ is separable, in the sense that it can be covered by countably many open sets homeomorphic to open subsets of separable Banach spaces. This can be proved by observing that $\mathscr{W}^{1,p}$ includes continuously into a $C^{0}$ space, and $C^{0}$ small neighborhoods are identified with open subsets of genuine $W^{1,p}$ vector spaces (using exponential type maps). Since $C^{0}$ spaces have a countable dense set of smooth elements, and $W^{1,p}$ vector spaces are separable, one concludes that $\mathscr{W}^{1,p}$ is separable.

Fix smooth data $H_{t},J_{t}$ admissible for a star-shaped domain $\Omega$. Let $\mathscr{H}$ be a Banach manifold of smooth Hamiltonians $H_{t}'$ which agree with $H_{t}$ on the complement of $\Omega$. We require $\mathscr{H}$ to be sufficiently large so that the image of the linearization of $H_{t}'\mapsto X_{H_{t}'}$ contains sections taking any chosen value at $p\in \Omega^{\mathrm{int}}$, supported in arbitrarily small balls around $p$. A good choice is Floer's $C^{\epsilon}$ space introduced in \cite[\S5]{floer-lag}, see also \cite[\S B]{wendl-sft}. One lets $\mathscr{H}$ be the space of $H_{t}'$ agreeing with $H_{t}$ on $\Omega^{c}$ and so that:
\begin{equation*}
  \sum\epsilon_{k}\norm{\nabla^{k}(H_{t}'-H_{t})}<\infty,
\end{equation*}
for a sequence $\epsilon_{k}>0$ decaying sufficiently quickly, where the norms are measured using any complete metric on $W$. This is separable, see \cite[Lemma B.4]{wendl-sft}.

Let $\mathscr{L}^{p}\to \mathscr{W}^{1,p}\times\mathscr{H}$ be the Banach manifold bundle whose fiber over $(u,\pi,\ell,H_{t}')$ consists of $L^{p}$ sections of $u^{*}TW$.

The nonlinear problem is described by the section $\mathscr{S}:\mathscr{W}^{1,p}\times \mathscr{H}\to \mathscr{L}^{p}$ given by:
\begin{equation*}
  \mathscr{S}(u,\pi,\ell,H_{t}')=\bd_{s}u+J_{t}(u)(\bd_{t}u-X_{t}'(u)),
\end{equation*}
where $\d\lambda(X_{t}',-)=-\d H'_{t}$. If $\mathscr{S}(u,\pi,\ell,H_{t}')=0$, then the linearized operator is well-defined, and is described as a map:
\begin{equation*}
  \mathrm{L}\mathscr{S}:T\mathscr{W}^{1,p}_{u,\pi,\ell}\oplus T\mathscr{H}_{H_{t}'}\to \mathscr{L}^{p}_{u}.
\end{equation*}

\begin{prop}
  The linearized operator $\mathrm{L}\mathscr{S}$ is surjective at every point $(u,\pi,\ell,H_{t}')$ in $\mathscr{S}^{-1}(0)$ (moreover, the linearized operator has a right inverse).
\end{prop}
\begin{proof}
  The argument is standard, see \cite[\S 8.3]{mcduffsalamon}, \cite[\S 5]{floer_hofer_salamon_transversality}. The restriction of $\mathrm{L}\mathscr{S}$ to the space of variations fixing $(\pi,\ell,H_{t}')$ is a Cauchy-Riemann operator $D$ on the cylinder $[0,\ell]\times \R/\Z$ with totally real boundary conditions. Since $u$ has many regular points $z_{0}$ in $u^{-1}(\Omega)$, and $\mathscr{H}$ is sufficiently large, the restriction of the linearized operator $\mathrm{L}\mathscr{S}$ to variations in $H_{t}'$ is transverse to the cokernel of any Cauchy-Riemann operator. Briefly, the reason is that the cokernel of $D$ is represented by the orthogonal complement to the image of $D$. Duality theory for partial differential operators implies this orthogonal complement is the kernel of an adjoint Cauchy-Riemann operator $D^{*}$. If the full linearization was not surjective, then we would conclude nonzero elements $\eta\in \ker D^{*}$ orthogonal to the linearization of $\mathscr{S}$ in directions tangent to $\mathscr{H}$. One then shows that $\eta$ would need to vanish at every regular point of $u$ in $u^{-1}(\Omega)$. Applying analytic continuation to $\eta$ implies that $\eta$ is identically zero, contradicting our assumption. That the operator has a right inverse follows from the fact $D$ is Fredholm.
\end{proof}

We want a surjective linearized operator for fixed $H_{t}$. The following result provides a way to achieve this.
\begin{lemma}
  Let $E_{i}$, $i=1,2,3$, be separable Banach spaces, and $F:E_{1}\times E_{2}\to E_{3}$ is a smooth map whose linearization has a right inverse at all points in $F^{-1}(0)$. Then $F^{-1}(0)$ is a smooth Banach manifold. Pick $(e_{1},e_{2})$ and let $G_{1}$ be the restriction of $F$ to $E_{1}\times \set{e_{2}}$ and $G_{2}$ the restriction of $\pr:E_{1}\times E_{2}\to E_{2}$ to $F^{-1}(0)$. Then the linearization of $G_{1}$ at $(e_{1},e_{2})$ is Fredholm if and only if the linearization of $G_{2}$ at $(e_{1},e_{2})$ is. In this case, the cokernel of the linearization of $G_{1}$ is canonically identified with the cokernel of the linearization of $G_{2}$.
  As a consequence, if $e_{2}$ is a regular value of $G_{2}$, then the linearization of $G_{1}$ is surjective.
\end{lemma}
\begin{proof}
  The argument is a straightforward analysis of Fredholm operators and some linear algebra. See \cite[\S 8.1]{wendl-sft}, \cite[Lemma A.3.6]{mcduffsalamon}.
\end{proof}

Applying the above result (locally) with $E_{1}=\mathscr{W}^{1,p}$, $E_{2}=\mathscr{H}$ and $E_{3}=\mathscr{L}^{p}$, one concludes that, if $H_{t}'$ is a regular value for the projection $\mathscr{S}^{-1}(0)\to \mathscr{H}$, then the linearization of $\mathscr{S}$ restricted to $\mathscr{W}^{1,p}\times \set{H_{t}'}$ will be surjective. The final ingredient is the Sard-Smale theorem which guarantees a Baire generic subset of regular values of the projection $\mathscr{S}^{-1}(0)\to \mathscr{H}$ (since it is a smooth map with Fredholm linearization between separable Banach manifolds), see \cite[\S A.5]{mcduffsalamon}. Thus we conclude:
\begin{prop}
  For any admissible data $J_{t},H_{t}$, there exist arbitrarily small perturbations $H_{t}'$ agreeing with $H_{t}$ outside $\Omega$ so that the moduli space $\mathscr{M}(J_{t},H_{t}')$ is a smooth manifold whose dimension equals the Fredholm index of the linearization of:
  \begin{equation*}
    (u,\pi,\ell)\in\mathscr{W}^{1,p}\mapsto \bd_{s}u+J_{t}(u)(\bd_{t}u-X_{t}(u)),
  \end{equation*}
  and this linearization is surjective.
\end{prop}
To complete the proof of the dimension part of Theorem \ref{theorem:main}, it remains to prove that this Fredholm index equals $\dim P+1$. This is the subject of the next section.

\subsection{Index theory}
\label{sec:index_theory}
As a first reduction, observe that the tangent space to $\mathscr{W}^{1,p}$ splits into the infinite dimensional piece consisting of variations which fix $\pi$ and $\ell$ and a finite dimensional complement of dimension $\dim P+1$. Thus it suffices to prove that the linearization restricted to variations fixing $\pi$ and $\ell$ has Fredholm index $0$.

\subsubsection{Travelling coordinates}
\label{sec:travelling_coordinates}

Fix $(u,\pi,\ell)\in \mathscr{M}(J_{t},H_{t})$. It is important to identify a neighborhood of $u$ with a linear space. A convenient way to do this is via \emph{travelling coordinates}.

Abbreviate $\Sigma=[0,\ell]\times \R$. For each $z\in \Sigma$, choose an open embedding: $$\varphi_{z}:B(1)^{n}\to M$$ so that $\pr\circ u(z)=\varphi_{z}(0)$. This can be constructed using Riemannian normal coordinates. Moreover, if $\mathfrak{B}(\pi,t)$ is orientable, we can arrange that $\varphi_{(s,t+1)}=\varphi_{(s,t)}$, while, if $\mathfrak{B}(\pi,t)$ is non-orientable, we can arrange that $\varphi_{(s,t+1)}=\varphi_{(s,t)}\rho$ where $\rho$ is a linear reflection. Let us be general and suppose that $\varphi_{(s,t+1)}=\varphi_{(s,t)}\rho$ for some orthogonal linear map $\rho$.

The maps $\varphi_{z}$ induce canonical embeddings: $$\Phi_{z}:B(1)\times \R^{n}\to T^{*}M\text{ so }\Phi_{z}^{*}\lambda=\sum p_{i}\d q_{i}.$$
There is a smooth map $w$ so that $u(z)=\Phi_{z}(w(z))$. The Floer equation for $u$ implies that $w$ satisfies the problem:
\begin{equation}\label{eq:w_eq}
  \left\{
    \begin{aligned}
      &\bd_{s}w+J_{z}(w)\bd_{t}w=A_{z}(w)\\
      &w(0,t)\in 0\times \R^{n}\\
      &w(\ell,t)\in B(1)\times 0\\
      &Rw(s,t+1)=w(s,t).
    \end{aligned}
  \right.
\end{equation}
Here $R$ is the linear symplectic extension of the orthogonal transformation $\rho$. This produces a linearized operator:
\begin{equation*}
  D(\eta)=\bd_{s}\eta+J(z)\bd_{t}\eta+B(z)\eta,
\end{equation*}
for $\eta\in W^{1,p}$ satisfying $\eta(0,t)\in 0\times \R^{n}$, $\eta(\ell,t)\in \R^{n}\times 0$ and $R\eta(s,t+1)=\eta(s,t)$.

It is standard that all such operators $D$ are Fredholm, see, e.g., \cite{mcduffsalamon,cant_thesis}. Since $J(z)$ is compatible with the standard symplectic structure on $\R^{2n}$, it can be deformed through the space of complex structures until $J(z)=J_{0}$ is the standard almost complex structure. This deformation does not change the Fredholm index. Similarly deform $B(z)$ until $B(z)=0$.
\begin{prop}\label{prop:index_0}
  The operator $D_{0}\eta=\bd_{s}\eta+J_{0}\bd_{t}\eta$ is an isomorphism between the spaces:
  \begin{enumerate}
  \item $\eta\in W^{1,p}$ with $\eta(0,t)\in 0\times \R^{n}$, $\eta(\ell,t)\in \R^{n}\times 0$ and $R\eta(s,t+1)=\eta(s,t)$, and
  \item $\xi \in L^{p}$ with $R\xi(s,t+1)=\xi(s,t)$.
  \end{enumerate}
\end{prop}
\begin{proof}
  Certainly $D_{0}$ is well-defined with the advertised domain (i) and codomain (ii). Liouville's theorem on the non-existence of bounded holomorphic functions implies $D_{0}$ is injective. Thus it remains only to prove that $D_{0}$ is surjective.

  It is a fact that $\kappa\mapsto D_{0}\kappa$ is surjective from $W^{1,p,\delta}([0,\ell]\times \R)\to L^{p,\delta}([0,\ell]\times \R)$, for small $\delta$, where the $\delta$ superscript denotes an exponentially weighted space, using the same boundary conditions for $W^{1,p,\delta}$. The argument is similar to \cite[\S 2.3]{salamon1997} and is proved in detail in \cite[Theorem 6.20]{cant_thesis}.

  Then, given $\xi$ satisfying $R\xi(s,t+1)=\xi(s,t)$, we can find $\mu$ compactly supported in $[0,\ell]\times \R$ so $\xi(s,t)=\sum_{k\in \Z}R^{k}\mu(s,t+k)$. Indeed, if $\beta(t)$ is a bump function and $\mu(s,t)=\beta(t)\xi(s,t)$, then the above will hold if:
  \begin{equation*}
    \sum_{k\in \Z}\beta(t+k)=1,
  \end{equation*}
  which can be achieved for $\beta$ supported in $(-1,1)$.

  Using the aforementioned surjectivity of $D_{0}$ on the exponentially weighted space, there is $\kappa\in W^{1,p,\delta}$ so $D_{0}\kappa=\mu$. Thus:
  \begin{equation*}
    \eta=\sum_{k\in \Z}R^{k}\kappa\implies D_{0}\eta=\xi,
  \end{equation*}
  and $\eta\in W^{1,p}$ will have the correct periodicity and boundary conditions. Indeed, this produces an explicit right inverse.
\end{proof}
Thus we have shown that the linearized operator has index $0$ when we restrict to variations fixing $\pi$ and $\ell$.

\subsection{Gluing}
\label{sec:gluing}
The goal in this section is to give an explicit description of the moduli space $\mathscr{M}(H_{t},J_{t})$ in the region where $\ell$ is very small. See \cite[\S 4.6]{abbondandolo_schwarz_2} for related discussion.

Introduce the notation $\mathscr{M}(\ell_{0})=\mathscr{M}\cap \set{\ell=\ell_{0}}$. Consider the two maps:
\begin{enumerate}
\item $\pi_{\ell_{0}}:\mathscr{M}(\ell_{0})\to P$ given by $(u,\pi,\ell_{0})\mapsto \pi$,
\item $\mathrm{ev}_{\ell_{0}}:\mathscr{M}(\ell_{0})\times \R/\Z\to M$ given by $(u,\pi,\ell_{0},t)\mapsto u(\ell_{0},t)$.
\end{enumerate}
The goals in this section are to prove that (a) sufficiently small $\ell_{0}$ are regular values for the projection $\ell:\mathscr{M}\to (0,\infty)$, (b) for $\ell_{0}$ sufficiently small, $\pi:\mathscr{M}(\ell_{0})\to P$ is a diffeomorphism and, (c) with respect to this identification, $\mathrm{ev}$ is $C^{0}$ close to $\mathfrak{B}$. The cobordism part of Theorem \ref{theorem:main} follows, since $C^{0}$ close maps are automatically cobordant, and it is clear that the slices $\mathscr{M}(\ell_{1})$ for large $\ell_{1}$ are cobordant to the slices for small $\ell_{0}$ (the moduli space is the cobordism).

\subsubsection{Travelling coordinates, revisited}
\label{sec:travelling_coordinates_revisited}

Fix $\pi_{0}\in P$, and for $\pi$ sufficiently close to $\pi_{0}$, let $\varphi_{t}^{\pi}$ be a path of embeddings $B(1)^{n}\to M$ so that $\varphi_{t}^{\pi}(0)=\mathfrak{B}(\pi,t)$ and $\varphi_{t+1}^{\pi}=\varphi_{t}^{\pi}\rho$ where $\rho:B(1)\to B(1)$ is a fixed orthogonal transformation (one can take $\rho$ to either be a reflection or the identity map, depending on whether $TM$ is orientable along the loop). Let $\Phi_{t}^{\pi}:B(1)\times \R^{n}\to W$ be the canonical extension of $\varphi_{t}^{\pi}$, and let $R$ be the linear symplectic extension of $\rho$.

Appealing to Proposition \ref{prop:small_ell}, for $\ell\le \ell_{0}$ sufficiently small, every element $(u,\pi,\ell)\in \mathscr{M}$ will satisfy $u(s,t)=\Phi_{t}^{\pi}(w(s,t)),$ where $w$ satisfies \eqref{eq:w_eq}. It is expedient to consider the conformal reparametrization $w(s,t)=v(\ell s,\ell t)$, where $v$ is defined on $[0,1]\times \R/\ell^{-1}\Z$, as in \S\ref{sec:length_shrinking_to_zero}. Then $v$ satisfies:
\begin{equation}\label{eq:v_eq}
  \left\{
    \begin{aligned}
      &\bd_{s}v+J_{t\ell}^{\pi}(v)\bd_{t}v=\ell A^{\pi}_{z\ell}(v)\\
      &v(0,t)\in 0\times \R^{n}\text{ and }v(1,t)\in B(1)\times 0\\
      &Rv(s,t+\ell^{-1})=v(s,t).
    \end{aligned}
  \right.
\end{equation}
Notice that as $\ell\to 0$, the equation for $v$ converges to the holomorphic curve equation.

\subsubsection{Linearized operator is uniformly an isomorphism}
\label{sec:uniformly_injective}
The non-linear problem \eqref{eq:v_eq} is posed for functions $v:[0,1]\times \R/\ell^{-1}\Z\to B(1)\times \R^{n}$, and hence we can linearize it at any map $v$, even those which do not solve \eqref{eq:v_eq}. We linearize by considering variations in the map $v$, and obtain the Cauchy-Riemann type operator:
\begin{equation*}
  D_{v,\pi,\ell}(\eta)=\bd_{s}\eta+J^{\pi}_{t\ell}(v)\bd_{t}\eta+B_{\pi,\ell}(z)\cdot \eta,
\end{equation*}
where $B_{\pi,\ell}(z)$ satisfies the estimate:
\begin{equation*}
  \abs{B_{\pi,\ell}(z)}\le C(\abs{v}+\abs{\bd_{t}v}+\ell),\text{ where }C=C(H_{t},J_{t},\mathfrak{B}).
\end{equation*}
A limiting argument yields:
\begin{prop}
  For sufficiently small $\gamma>0$ and $\pi$ nearby $\pi_{0}$ there exists $\delta>0$ so that:
  \begin{equation*}
    \norm{v}_{W^{1,p}}+\ell<\delta\implies \gamma\norm{\eta}_{W^{1,p}}\le \norm{D_{v,\pi,\ell}(\eta)}_{L^{p}}
  \end{equation*}
  In other words, the linearized operator $D_{v,\pi,\ell}$ is uniformly injective.
\end{prop}
\begin{proof}
  For each $\tau\in \R/\Z$, consider the limiting linearized problem:
  \begin{equation*}
    \left\{
      \begin{aligned}
        &D^{\infty}_{\pi,\tau}(\eta)=\bd_{s}\eta+J^{\pi}_{\tau}(0)\bd_{t}\eta\\
        &\eta\in W^{1,p}([0,1]\times \R,\R^{2n})\\
        &\eta(0,t)\in 0\times \R^{n}\text{ and }\eta(1,t)\in \R^{n}\times 0.
      \end{aligned}
    \right.
  \end{equation*}
  As in the proof of Proposition \ref{prop:index_0}, this is an isomorphism to $L^{p}([0,1]\times \R,\R^{2n})$, since the two boundary conditions are transverse totally real-subspaces. In particular, we can pick $\gamma$ large enough so that:
  \begin{equation*}
    6\gamma\norm{\eta}_{W^{1,p}}\le \norm{D^{\infty}_{\pi,\tau}(\eta)}_{L^{p}}.
  \end{equation*}
  for all $\pi$ near $\pi_{0}$ and $\tau\in \R/\Z$. This fixes $\gamma$.

  Suppose the statement does not hold for any $\delta$. Then there exists $v_{n},\pi_{n},\ell_{n}$ so that $\norm{v_{n}}_{W^{1,p}}+\ell_{n}$ converges to zero and $\pi_{n}$ converges, but $D_{n}:=D_{v_{n},\pi_{n},\ell_{n}}$ fails the desired estimate (for some $\eta_{n}$).

  Suppose, as a special case, that $\eta_{n}(s,t)$ is supported in a region where $t\in I_{n}$ and $I_{n}$ has length $3L$. Let $t^{*}_{n}$ be the center of $I_{n}$. By picking $\tau_{n}=t_{n}^{*}\ell_{n}$, we conclude that, for any $\theta$, $$\norm{D_{n}(\eta_{n})-D^{\infty}_{\pi_{n},\tau_{n}}(\eta_{n})}_{L^{p}}\le \theta\norm{\eta_{n}}_{W^{1,p}},$$ for $n$ large enough. This follows from the computation:
  \begin{equation*}
    D_{n}(\eta_{n})-D^{\infty}_{\pi_{n},\tau_{n}}(\eta_{n})=(J_{t\ell_{n}}^{\pi_{n}}(v_{n})-J_{\tau}^{\pi_{n}}(0))\bd_{t}\eta_{n}+B_{\pi,\ell}(z)\eta_{n},
  \end{equation*}
  appealing to standard estimates for the $W^{1,p}$ norm, and using $\abs{t\ell_{n}-\tau}\le 3\ell_{n}$. How large $n$ needs to be depends on $L,\theta$, how quickly $\norm{v_{n}}_{W^{1,p}}+\ell_{n}$ decays, and on the derivatives of $\tau,v\mapsto J^{\pi_{n}}_{\tau}(v)$, but is otherwise independent of $\eta_{n}$. Then:
  \begin{equation*}
    (6\gamma-\theta)\norm{\eta_{n}}_{W^{1,p}}\le \norm{D_{n}(\eta_{n})}_{L^{p}}.
  \end{equation*}
  In general, cover the domain of $\eta_{n}$ by regions $[0,1]\times I$ where $I$ has length $3L$ so that each point is contained in at most four of the intervals\footnote{If $\ell_{n}^{-1}$ is an integer multiple of $L/3$, one can replace ``four'' by ``three.''} and at least one middle third. Using translations of bump functions $\beta$ which are $1$ on $[L,2L]$ and supported in $[0,3L]$, conclude that:
  \begin{equation*}
    (6\gamma-\theta)\norm{\beta\eta_{n}}_{W^{1,p}(I)}\le \norm{\beta D_{n}(\eta_{n})}_{L^{p}(I)}+CL^{-1}\norm{\eta_{n}}_{L^{p}(I)},
  \end{equation*}
  where we assume that the first derivative of $\beta$ is bounded by $L^{-1}$, and where $C$ depends only on the values of the complex structure $J_{t}^{\pi}$. Sum over all the intervals to conclude:
  \begin{equation*}
    (6\gamma-\theta)\norm{\eta_{n}}_{W^{1,p}}\le 4\norm{D_{n}(\eta_{n})}_{L^{p}}+4CL^{-1}\norm{\eta_{n}}_{L^{p}}.
  \end{equation*}
  Make $L$ large enough that $4CL^{-1}<\gamma$, and pick $\theta<\gamma$. For $n$ sufficiently large:
  \begin{equation*}
    \gamma\norm{\eta_{n}}_{W^{1,p}}\le \norm{D_{n}(\eta_{n})}_{L^{p}},
  \end{equation*}
  contradicting the assumption.
\end{proof}

As a corollary, we conclude that, for $\norm{v}_{W^{1,p}}+\ell<\delta$, the linearization $D_{v,\pi,\ell}$ has a bounded inverse (since it is injective and has Fredholm index $0$), and that the norm of the inverse is bounded by $1/\gamma$.

\subsubsection{Existence, uniqueness, and the implicit function theorem}
\label{sec:existence_uniqueness}

\begin{prop}\label{prop:small_w1p}
  For sufficiently small $\epsilon>0$, $\ell>0$, and $\pi$ nearby $\pi_{0}$, there exists a unique solution of \eqref{eq:v_eq} in the $W^{1,p}$ ball of radius $\epsilon$ around $0$. This solution depends smoothly on $\ell$ and $\pi$.
\end{prop}
\begin{proof}
  Uniqueness and smooth dependence follows from implicit function theorem in Banach spaces, and the existence follows from Newton's iteration, as explained in \cite[\S A.3]{mcduffsalamon}.

  For existence and uniqueness (with fixed $\ell,\pi$), one uses \S\ref{sec:uniformly_injective} to ensure the hypotheses of \cite[Proposition A.3.4]{mcduffsalamon} (with $c_{\mathrm{MS}}=\gamma^{-1}$, $\delta_{\mathrm{MS}}=\epsilon$, $x_{0,\mathrm{MS}}=x_{1,\mathrm{MS}}=0$, where the subscript $\mathrm{MS}$ signifies terms in the statement of their proposition). The details are left to the reader.
\end{proof}

Returning to the global problem, we conclude that, for $\ell_{0}$ sufficiently small, the projection map $\mathscr{M}(\ell_{0}):=\mathscr{M}(H_{t},J_{t},\mathfrak{B})\cap \set{\ell=\ell_{0}}\to P$ is surjective and has a locally defined smooth right inverse (using the smoothly varying solutions guaranteed by Proposition \ref{prop:small_w1p}).

However, Proposition \ref{prop:small_ell} and the derivative bounds on the conformal reparametrization from \S\ref{sec:length_shrinking_to_zero} imply that every solution in $\mathscr{M}(\ell_{0})$ enters the $W^{1,p}$ ball of radius $\epsilon$ around $0$, after applying the coordinate changes and conformal reparametrizations needed to pass from $u$ to $v$. Thus the local smooth right inverses $P\to \mathscr{M}(\ell_{0})$ patch together to be a global smooth inverse.

Since the $W^{1,p}$ norm controls the $C^{0}$ norm, $\mathrm{ev}_{\ell_{0}}:\mathscr{M}(\ell_{0})\times \R/\Z\to M$ is $C^{0}$ close to $\mathfrak{B}$, with respect to the identification of $\mathscr{M}(\ell_{0})$ with $P$, so $\mathrm{ev}_{\ell_{0}}$ is cobordant to $\mathfrak{B}$. Since the mapping degree of $\mathscr{M}(\ell_{0})\to P$ is stable under cobordisms, the proof of Theorem \ref{theorem:main} is complete.

\bibliography{citations}
\bibliographystyle{alpha}
\end{document}